\documentclass[12pt,a4paper]{amsart}
\newcommand{\guio}[1]{\nobreakdash-\hspace{0pt}#1}

\usepackage{amsmath,amsthm}
\usepackage{amssymb}

\newcommand{\C}{{\mathbb C}}       
\newcommand{\R}{{\mathbb R}}       
\newcommand{\Z}{{\mathbb Z}}       
\newcommand{\DD}{{\mathcal D}}
\newcommand{\HH}{{\mathcal H}}
\newcommand{\LL}{{\mathcal L}}
\newcommand{\D}{{\Delta}}
\newcommand{\AZ}{{\mathcal A}}

\newcommand{\CC}{{\mathcal C}}

\newcommand{\diam}{{\rm d}}
\newcommand{\dist}{{\rm dist}}
\newcommand{\ds}{\displaystyle }

\newcommand{\fiproof}{{\hspace*{\fill} $\square$ \vspace{2pt}}}

\newcommand{\pv}{{\rm p.v.}}

\newcommand{\rf}[1]{{(\ref{#1})}}
\newcommand{\supp}{{\rm supp}}

\newcommand{\vphi}{{\varphi}}
\newcommand{\ve}{{\varepsilon}}

\newcommand{\vvv}{{\vspace{3mm}}}
\newcommand{\wt}[1]{{\widetilde{#1}}}
\newcommand{\wh}[1]{{\widehat{#1}}}

\newcommand{\noi}{\noindent}

\newtheorem{theorem}{Theorem}[section]
\newtheorem*{theorema*}{Theorem A}
\newtheorem*{theoremb*}{Theorem B}
\newtheorem*{theoremc*}{Theorem C}
\newtheorem*{theoremd*}{Theorem D}
\newtheorem{lemma}[theorem]{Lemma}

\newtheorem{claim}[theorem]{Claim}
\theoremstyle{definition}
\newtheorem{definition}[theorem]{Definition}

\theoremstyle{remark}
\newtheorem{remark}[theorem]{\bf Remark}

\numberwithin{equation}{section}

\begin{document}

\title[Uniform rectifiability and CZO's]{Uniform rectifiability,
Calder\'{o}n-Zygmund operators with odd kernel, and quasiorthogonality}

\author[XAVIER TOLSA]{Xavier Tolsa}

\address{Instituci\'o Catalana de Recerca i Estudis Avan\c{c}ats (ICREA) and Departament de Ma\-te\-m\`a\-ti\-ques, Universitat Aut\`onoma de
Bar\-ce\-lo\-na, 08193 Bellaterra (Barcelona), Catalonia}

\email{xtolsa@mat.uab.cat}

\thanks{Partially supported by grants MTM2007-62817 and
 and 2005-SGR-00744 (Generalitat
de Catalunya)}


\date{May, 2008.}

\begin{abstract}
In this paper we study some questions in connection with uniform rectifiability and the $L^2$~boundedness
of Calder\'on-Zygmund operators. We show that uniform rectifiability can be characterized in terms of
some new adimensional coefficients which are related to the Jones' $\beta$ numbers. We also use these new coefficients 
to prove that $n$-dimensional Calder\'on-Zygmund operators with odd kernel of type $\CC^2$ are bounded
in $L^2(\mu)$ if $\mu$ is an $n$-dimensional uniformly rectifiable measure. 
\end{abstract}

\maketitle


\section{Introduction}\label{secintro}

In this paper we study some questions in connection with uniform rectifiability and the $L^2$~boundedness
of Calder\'on-Zygmund operators.

Given $0<n\leq d$, we say that a Borel measure $\mu$ on $\R^d$ is $n$-dimensional Ahlfors-David regular, or simply 
AD regular, if there exists 
some constant $C_0$ such that $C_0^{-1}r^n\leq\mu(B(x,r))\leq C_0r^n$ for all $x\in\supp(\mu)$, $0<r\leq\diam(\supp(\mu))$.
It is not difficult to see that such a measure~$\mu$ must be of the form $d\mu=\rho\,d\HH^n_{|\supp(\mu)}$, where
 $\rho$ is some positive function bounded from above and from below and $\HH^n$ stands for the $n$-dimensional
Hausdorff measure. A Borel set $E\subset \R^d$ is called AD regular if the measure $\HH^n_{|E}$ is AD regular.

Throughout all the paper $\mu$ will be an $n$-dimensional AD regular measure on $\R^d$, with $n$~integer and $0<n\leq d$.

The notion of uniform $n$-rectifiability (or simply, uniform rectifiability) was introduced by David 
and Semmes in \cite{DS2}.
For $n=1$, an AD regular $1$-dimensional measure is uniformly rectifiable if its support is contained in an AD regular curve.
For an arbitrary integer $n\geq1$, the notion is more complicated. One of the many equivalent definitions (see Chapter
 I.1 of \cite{DS2}) is the following: $\mu$ is uniformly rectifiable if there exist $\theta,M>0$ so that, for
each $x\in\supp(\mu)$ and $R>0$, there is a Lipschitz mapping $g$ from the $n$-dimensional ball $B_n(0,R)\subset\R^n$
into $\R^d$ such that $g$ has Lipschitz norm $\leq M$ and
$$\mu\bigl(B(x,R)\cap g(B_n(0,R))\bigr) \geq \theta R^n.$$
In the language of \cite{DS2}, this means that {\em $\supp(\mu)$ has big pieces of Lipschitz images of $\R^n$}.ç
A Borel set $E\subset\R^d$ is called uniformly rectifiable if $\HH^n_{|E}$ is uniformly rectifiable.

The $n$-dimensional Riesz transform of a function $f:\R^d\to\R$ with respect to $\mu$ is
$$R_\mu f(x) = \int \frac{x-y}{|x-y|^{n+1}}\,f(y)\,d\mu(y),$$
for $x\not\in\supp(\mu)$. Notice that $(x-y)/|x-y|^n$ is a vectorial kernel.
In \cite{DS1} it is proved that if $\mu$ is uniformly rectifiable, then $R_\mu$ is bounded in $L^2(\mu)$ (see \rf{eq**} below
for the precise definition of $L^2$ boundedness of $R_\mu$). On the other hand, it is an open problem
if, given an $n$-dimensional 
AD~regular measure $\mu$, with $n>1$, 
the $L^2(\mu)$ boundedness of the $n$-dimensional Riesz transform implies the uniform rectifiability of $\mu$.
See \cite[Chapter 7]{Pajot}. 
 This problem
has only been solved in the case $n=1$ (by Mattila, Melnikov and Verdera \cite{MMV}), by using the notion of 
curvature of measures, which is useful only for~$n=1$ (see \cite{Farag}). 
In fact, there is a strong connection between this question for $n=1$ and the so called Painlev\'e problem (i.e.\ the problem of characterizing
removable singularities for bounded analytic function in a geometric way). See \cite{David-revista}, \cite{Leger}, \cite{Tolsa-sem}, and
\cite{Volberg}, for example. 
In the present paper we develop new techniques and we 
obtain some results in connection with the problem of $L^2$ boundedness of Riesz transforms and 
rectifiability.

A basic tool for the study of uniform rectifiability are the coefficients $\beta_p$.
Given $1\leq p < \infty$ and a cube $Q$, one sets
$$\beta_p(Q) = \inf_L 
\biggl\{ \frac1{\ell(Q)^n}\int_{2Q} \biggl(\frac{\dist(y,L)}{\ell(Q)}\biggr)^pd\mu(y)\biggr\}^{1/p},$$
where the infimum is taken over all $n$-planes in $\R^d$ and $\ell(Q)$ denotes the side length of $Q$.
For $p=\infty$ one has to replace the $L^p$ norm by a supremum:
$$\beta_\infty(Q) = \inf_L \biggl\{ \sup_{y\in \supp(\mu)\cap 2Q}
\frac{\dist(y,L)}{\ell(Q)}\biggr\},$$
where the infimum is taken over all $n$-planes $L$ in $\R^{d}$ again.
The coefficients $\beta_p$ first appeared in \cite{Jones-Escorial} and \cite{Jones}, in the case $n=1$, 
$p=\infty$. In
 \cite{Jones-Escorial} P.~Jones showed, among other results, how the $\beta_\infty$'s can be used to prove the $L^2$
boundedness of the Cauchy transform on Lipschitz graphs. In \cite{Jones}, he characterized $1$-dimensional 
uniformly rectifiable sets in terms of the~$\beta_\infty$'s. He also obtained other quantitative results on 
rectifiability without the AD regularity assumption. For other $p$'s and $n\geq1$, the $\beta_p$'s were introduced
by David and Semmes in their pioneering study of uniform rectifiability in \cite{DS1}.

In the present paper we will define other coefficients, in the spirit of the $\beta_p$'s, 
which are also useful for the study of uniform rectifiability.
Before introducing these coefficients, we need to define a metric on the space of
finite Borel measures (supported in a ball).
Given a closed ball $B\subset \R^d$ and two finite Borel measures $\sigma$, $\nu$
on $\R^d$ , we set
$$\dist_B(\sigma,\nu):= \sup\Bigl\{ \Bigl|{\textstyle \int f\,d\sigma  -
\int f\,d\nu}\Bigr|:\,{\rm Lip}(f) \leq1,\,\supp(f)\subset
B\Bigr\},$$
where ${\rm Lip}(f)$ stands for the Lipschitz constant of $f$.
It is easy to check that this is indeed a distance in the space of finite Borel measures supported in the interior of 
$B$. See \cite[Chapter 14]{Mattila-llibre} for other properties of this distance.

Given an AD regular measure $\mu$ on $\R^d$ and a cube $Q$ which intersects $\supp(\mu)$, 
we consider the closed ball $B_Q\!:=\! B(z_Q,3\,\diam(Q))$, where $z_Q$ and $\diam(Q)$ stand  for the center 
and diameter of $Q$, respectively. Then we define
$$
\alpha_\mu^n(Q) := \frac1{\ell(Q)^{n+1}}\,\inf_{c\geq0,L} \,\dist_{B_Q}(\mu,\,c\HH^n_{|L}),$$
where the infimum is taken over all the constants $c\geq0$ and all the $n$-planes $L$. 
For convenience, if $Q$ does not intersect $\supp(\mu)$, we set $\alpha^n_\mu(Q)=0$.
To simplify notation,
we will also write $\alpha(Q)$ instead of $\alpha_\mu^n(Q)$.

Notice that the coefficient $\alpha(Q)$ measures, in a scale invariant way, how close is $\mu$ to a flat $n$-dimensional measure in $B_Q$. Recall that a measure $\nu$ is said to be flat and $n$-dimensional if it is of the 
form $\nu=c\HH^n_{|L}$, for some constant $c>0$ and some $n$-plane $L$.
It is worthwile to compare the coefficients $\alpha$  with the $\beta_p$'s: basically, the latter coefficients only give information on how close
$\supp(\mu)\cap 2Q$ is to some $n$-plane (more precisely, how close is $\supp(\mu)\cap 2Q$ to be contained in some $n$-plane). 
On the other hand, the coefficients $\alpha$ contain more information than $\beta_p$. 
For instance, if $\supp(\mu)$ is contained in an $n$-plane, then $\beta_p(Q)=0$ for any $Q$. However, in this case we may still have $\alpha(Q)>0$. This will be the case if $\mu$ does not coincide with a flat a measure in $B_Q$.
In Lemma \ref{lempr1} we will show that, for $\mu$ AD regular,  
 $$\beta_1(Q)\leq C\alpha(Q)$$
for all dyadic cubes $Q\in \DD$ (see Section \ref{secprelim} for the precise definition of the dyadic cubes from $\DD$ in the context of AD regular
measures).
 As the preceding example shows, the opposite inequality is false in general.

In Section \ref{secrectif} we will prove the following result:

\begin{theorem}\label{teolips}
Consider the $n$-dimensional Lipschitz graph $\Gamma:=\{(x,y)\in \R^n \times \R^{d-n}:\,y=A(x)\}$, with
$\|\nabla A\|_\infty\leq C_1<\infty$, and let $d\mu(z)= \rho(z)\,d\HH^n_{|\Gamma}(z)$, where $\rho(\cdot)$ is a 
function such that $0\leq\rho(z)\leq C_1$ uniformly on $z\in\Gamma$. 
Then, the coefficients $\alpha$ satisfy the following Carleson packing condition:
\begin{equation}\label{eqpac6}
\sum_{Q\in \DD_{\R^d}(R)}\alpha(Q)^2\,\mu(Q) \leq C_2 \ell(R)^n,
\end{equation}
for any cube $R\subset\R^d$ which intersects $\supp(\mu)$, where $C_2$ depends only on $n,d$ and $C_1$.
\end{theorem}

For simplicity, in this theorem we assume that $R$ has sides parallel to the axes. We have denoted
by $\DD_{\R^d}(R)$ the collection of dyadic cubes generated by $R$, i.e. the collection of cubes contained 
in $R$ which are obtained by splitting $R$ dyadically.

Recall that \rf{eqpac6} also holds if one replaces the coefficients $\alpha(Q)$ by $\beta_p(Q)$, for $1\leq p\leq 2n/(n-2)$, as shown in \cite{Dorronsoro}.

We will also see in Section \ref{secrectif} that uniformly rectifiable sets can be characterized in terms of the
$\alpha$'s, similarly to what happens with the $\beta_p$'s and the so called bilateral $\beta_p$'s:

\begin{theorem}\label{teounif}
Let $\mu$ be an $n$-dimensional AD regular measure. The following are equivalent:
\begin{itemize}
\item[(a)] $\mu$ is uniformly rectifiable.
\item[(b)] For any dyadic cube $R\in\DD$,
\begin{equation}\label{pack0}
\sum_{Q\in \DD:Q\subset R}\alpha(Q)^2\,\mu(Q) \leq C \mu(R),
\end{equation}
with $C$ independent of $R$.

\item[(c)] For all $\ve>0$, there exists some constant $C(\ve)$ such that the collection ${\mathcal B}_\ve$ of those cubes $Q\in\DD$ such that $\alpha(Q)>\ve$ satisfies
$$\sum_{Q\in \mathcal{B}_\ve: Q\subset R} \mu(Q)\leq C(\ve)\mu(R),$$
for any cube $R\in\DD$.
\end{itemize}
\end{theorem}

In the theorem, $\DD$ stands for the lattice of dyadic cubes associated to $\mu$ which is described in Section 
\ref{secprelim}.

Our main motivation to introduce the coefficients $\alpha(\cdot)$ is to study the relationship between uniform
rectifiability and the $L^2(\mu)$ boundedness of Calder\'on-Zygmund operators. In particular, we think that they can be a useful tool to study the
aforementioned problem of proving that the $L^2(\mu)$ boundedness of the Riesz transform $R_\mu$ implies the uniform rectifiability of $\mu$ when $\mu$ is AD regular, as well as other related problems (see \cite{Tolsa-jfa} for a recent application concerning the existence of principal values for Riesz transforms and rectifiability).
The kernels $K(\cdot):
\R^d\setminus\{0\}\to\R$ that we will consider satisfy
\begin{equation} \label{satis0}
|\nabla^j K(x)|\leq \frac{C}{|x|^{n+j}}\qquad \mbox{for $0\leq j\leq 2$ and $x\in\R^d\setminus\{0\}$,}
\end{equation}
and moreover $K(-x)=-K(x)$, for all $x\neq0$ (i.e.\ they are odd). The kernel $x/|x|^{n+1}$ of the $n$-dimensional Riesz transform
is a basic example (to be precise, we should consider the scalar components $x_i/|x|^{n+1}$.)

Given a finite positive or real Borel measure $\nu$, we define
$$T\nu(x) := \int K(x-y)\,d\nu(y), \qquad\mbox{for $x\in\R^d\setminus\supp(\nu)$}.$$
We say that $T$ is an $n$-dimensional  Calder\'{o}n-Zygmund operator (CZO) with kernel
$K(\cdot)$. The integral in the definition may not be
absolutely convergent for $x\in\supp(\nu)$. For this reason, we
consider the following $\ve$-truncated operators $T_\ve$, $\ve>0$:
$$T_\ve\nu(x) := \int_{|x-y|>\ve} K(x-y)\,d\nu(y), \qquad{x\in\R^d}.$$
Observe that now the integral on the right hand side above is absolutely convergent. We also denote
$$T_* \nu(x) = \sup_{\ve>0}|T_\ve \nu(x)|,$$
and
$$\pv T\nu(x) = \lim_{\ve\to0} T_\ve\nu(x),$$
whenever the limit exists.

If $\mu$ is a fixed positive Borel measure and $f\in
L^1_{loc}(\mu)$, we set
$$T_\mu f(x) := T(f\,d\mu)(x),\qquad \mbox{for $x\in\R^d\setminus\supp(f\,d\mu),$}$$
and
\begin{equation}\label{eq**}
T_{\mu,\ve} f(x) := T_\ve(f\,d\mu)(x).
\end{equation}
The last definition makes sense for all $x\in\R^d$ if, for
example, $f\in L^1(\mu)$. We say that $T_\mu$ is bounded on
$L^2(\mu)$ if the operators $T_{\mu,\ve}$ are bounded on
$L^2(\mu)$ uniformly on $\ve>0$.

In Sections \ref{secz1} and \ref{sec6} of this paper we will prove the following result:

\begin{theorem} \label{teomain}
Let $\mu$ be an $n$-dimensional AD regular measure on $\R^d$ and $T$ an $n$\guio{dimensional}~CZO associated to an odd 
kernel $K(\cdot):\R^d\setminus\{0\}\to\R$ satisfying \rf{satis0}.
Then we have
\begin{equation}\label{pral11}
\|T_*\mu\|_{L^2(\mu)}^2\lesssim \sum_{Q\in\DD} \alpha(Q)^2\mu(Q) + \mu(\C).
\end{equation}
If $\sum_{Q\in\DD} \alpha(Q)^2\mu(Q)<\infty$, then $\pv T\mu(x)$ exists for $\mu$-a.e. $x\in \R^d$ and
\begin{equation}\label{pral110}
\|\pv T\mu\|_{L^2(\mu)}^2\lesssim \sum_{Q\in\DD} \alpha(Q)^2\mu(Q).
\end{equation}
If $\mu$ is uniformly rectifiable, then $T_\mu$ is bounded in $L^2(\mu)$.
\end{theorem}

See Section \ref{secprelim} for the notation $\lesssim$.

The fact that uniform rectifiability implies the $L^2$ boundedness of CZO's with odd kernel
was already known for $\CC^\infty$ kernels satisfying
\begin{equation}\label{satis11}
|\nabla^j K(x)|\leq \frac{C(j)}{|x|^{n+j}}
\end{equation} 
for {\em all} $j\geq0$ (and maybe also assuming \rf{satis11} only for a finite but big number of $j$'s). 
See~\cite{David-surfaces} and Section II.6.B of \cite{David-LNM}. However, 
the result is new if one only asks~\rf{satis11} for $0\leq j\leq 2$, and so it improves on previous
results. 


Most proofs of the $L^2$ boundedness of CZO's (with $\CC^\infty$ kernel) with respect to uniformly rectifiable 
measures use the method of rotations
and the $L^2$ boundedness of the Cauchy transform on Lipschitz graphs (in fact, all proofs known by the author).
This is not the case with the arguments that we use in this paper.
Roughly speaking, our basic idea consists in 
decomposing $T\mu$
dyadically, and in obtaining estimates by comparing on each cube $Q$ the measure $\mu$ with the flat measure that 
minimizes $\alpha(Q)$ (notice that if $\nu$ is a flat measure then $T_\ve \nu$ vanishes on $\supp(\nu)$).
This idea is inspired in part by the proof
 of the $L^2$ boundedness of the Cauchy transform
on Lipschitz graphs by P. Jones in \cite{Jones-Escorial}. However, we recall that the arguments in 
\cite{Jones-Escorial} only work for the Cauchy transform.

Given a non-increasing radial $\CC^2$ function $\psi$ such
that $\chi_{B(0,1/2)}\leq \psi\leq \chi_{B(0,2)}$, for each
$j\in\Z$, we set $\psi_j(z) := \psi(2^jz)$ and
$\vphi_j:=\psi_j-\psi_{j+1}$,  so that each function $\vphi_j$ is
non-negative and supported in the annulus $A(0,2^{-j-2},2^{-j+1})$, 
and moreover we have $\sum_{j\in\Z}\vphi_j(x) = 1$
for any $x\in\R^d\setminus \{0\}$. For
each $j\in Z$ we denote $K_j(x) = \vphi_j(x)\,K(x)$ and
\begin{equation}\label{eqaux10}
T_j \mu(x) = \int K_j(x-y)\,d\mu(y).
\end{equation}
Notice that, at a formal level, we have $ T\mu = \sum_{j\in\Z} T_j\mu,$ and so
$$\|T\mu\|_{L^2(\mu)}^2 = \sum_{j\in \Z}\|T_j\mu\|_{L^2(\mu)}^2 + \sum_{j\neq k}\langle T_j\mu,\,T_k\mu\rangle.$$
To prove Theorem \ref{teomain} we will show that both sums in the right hand side above are bounded above by
$\sum_{Q\in\DD} \alpha(Q)^2\mu(Q)$. In fact, to show that 
\begin{equation}\label{eqqo1}
\sum_{j\in \Z}\|T_j\mu\|_{L^2(\mu)}^2 \lesssim \sum_{Q\in\DD} \alpha(Q)^2\mu(Q)
\end{equation}
is quite easy (see Lemma \ref{lemtq2} below), while the estimate of $\sum_{j\neq k}\langle T_j\mu,\,T_k\mu\rangle$
requires much more work.

The final part of this paper deals with Riesz transforms. 
Let $R_j$ denote the doubly truncated Riesz transform associated to the kernel $\vphi_j(x)\,x/|x|^{n+1}$, like in \rf{eqaux10}.
In Section \ref{secqo} we will prove the following:

\begin{theorem}\label{teoqo}
Let $\mu$ be an $n$-dimensional AD regular measure on $\R^d$. For $j\in\Z$, let $R_j$ denote the doubly 
truncated Riesz transform introduced in Definition \ref{defpsifi}. For any $Q\in\DD$, we have
\begin{equation}\label{eqteoqo}
\sum_{P\in \DD:P\subset Q}\beta_2(P)^2\mu(P)\leq C_3 \Bigl(\sum_{j\in\Z} \|R_j\mu_{|3Q}\|_{L^2(\mu)}^2 + \mu(Q)
\Bigr),
\end{equation}
Moreover, $C_3$ only depends on $n,d,C_0$ and the constants involving $\psi$
in Definition \ref{defpsifi}.
Therefore, if for any cube $Q\in\DD$,
\begin{equation}\label{eqqo9}
\sum_{j\in\Z} \|R_j\mu_{|Q}\|_{L^2(\mu)}^2 \lesssim\mu(Q),
\end{equation}
then $\mu$ is uniformly rectifiable.
\end{theorem}

Let us remark that the kernels of the doubly truncated Riesz transforms $R_j$ introduced in Definition~\ref{defpsifi}
are defined as we did just above \rf{eqaux10}, although some additional properties are required for the auxiliary 
functions $\vphi_j$.
Notice also that the estimate \rf{eqteoqo} can be understood as a kind of converse of inequality \rf{eqqo1}.

Mattila and Preiss \cite[Theorem 5.5]{MPr} have already proved  
that if all the CZO's with kernel of the form $\vphi(|x|) x/|x|^{n+1}$ satisfying \rf{satis11} are bounded in 
$L^2(\mu)$, then $\mu$ is $n$-rectifiable, i.e. there exist Lipschitz mappings $g_i:\R^n\to\R^d$ such that 
$\mu\Bigl(\R^n\setminus \bigcup_{i=1}^\infty g_i(\R^n)\Bigr) = 0.$
David and Semmes \cite[Theorem 2.59]{DS2} have shown that, moreover, $\mu$ is uniformly rectifiable. It is not difficult to show that
the assumption that all CZO's with kernel of the form $\vphi(|x|) x/|x|^{n+1}$ satisfying \rf{satis11} are bounded in 
$L^2(\mu)$ implies 
that \rf{eqqo9} holds (see \cite[Chapter 3]{DS1} for a related argument). 

\vvv
{\em Acknowledgement:} The author wants to thank the referee for helpful comments that improve on the readability of the paper.


\section{Preliminaries}\label{secprelim}

As usual, in the paper the letter `$C$' stands for some
constant which may change its value at different occurrences, which quite often depends only on $n$, $d$, and
the AD-regularity constant $C_0$. On
the other hand, constants with subscripts, such as $C_1$, retain
its value at different occurrences. The notation $A\lesssim B$
means that there is some fixed constant $C$ such that
$A\leq CB$, with $C$ as above. Also, $A\approx B$ is equivalent to $A\lesssim B
\lesssim A$.

The closed ball centered at $x$ with radius $r$ is denoted by $B(x,r)$.

Recall that throughout all the paper we are assuming that $\mu$ is a fixed AD regular $n$\guio{dimensional} measure. We
denote $E=\supp(\mu)$.
For simplicity of notation, we will also assume that $\diam(E)=\infty$. However all the results stated in the paper 
are valid without this assumption. The required modifications are minimal.

In this paper we will use the so called ``dyadic cubes'' built in \cite[Appendix 1]{David-LNM} (see also~\cite{Christ} 
for an alternative construction). 
These dyadic cubes are not true cubes, but they play this role with respect to $\mu$, in a sense. 

Let us explain 
which are the precise results and properties about our lattice of dyadic cubes. 
For each $j\in\Z$, one can construct a family $\DD_j$ of Borel subsets of $E$ (the dyadic cubes of the $j$-th
 generation) such that:
\begin{itemize}
\item[(i)] each $\DD_j$ is a partition of $E$, i.e.\ $E=\bigcup_{Q\in \DD_j} Q$ and $Q\cap Q'=\varnothing$ whenever $Q,Q'\in\DD_j$ and
$Q\neq Q'$;
\item[(ii)] if $Q\in\DD_j$ and $Q'\in\DD_k$ with $k\leq j$, then either $Q\subset Q'$ or $Q\cap Q'=\varnothing$;
\item[(iii)] for all $j\in\Z$ and $Q\in\DD_j$, we have $2^{-j}\lesssim\diam(Q)\leq2^{-j}$ and $\mu(Q)\approx 2^{jn}$;
\item[(iv)] if $Q\in\DD_j$, there is a point $z_Q\in Q$ (the center of $Q$) such that $\dist(z_Q,E\setminus Q)
\gtrsim 2^{-j}$.
\end{itemize}
We denote $\DD=\bigcup_{j\in\Z}\DD_j$. Given $Q\in\DD_j$, the unique cube $Q'\in\DD_{j-1}$ which contains $Q$ is called the parent of $Q$.
We say that $Q$ is a sibling of $Q'$. If $Q$ is from the generation $j$, we write $J(Q)=j$.

For $Q\in \DD_j$, we define the side length
 of $Q$ as $\ell(Q)=2^{-j}$. Notice that $\ell(Q)\lesssim\diam(Q)\leq \ell(Q)$.
Actually it may happen that a cube $Q$ belongs to $\DD_ j\cap \DD_k$ with $j\neq k$, because there may exist cubes
with only one sibling. In this case, $\ell(Q)$ is not well defined. However this problem can be solved in many ways.
For example, the reader may think that a cube is not only a subset of $E$, but a couple $(Q,j)$, where $Q$ is
a subset of $E$ and $j\in\Z$ is such that $Q\in\DD_j$.

Given $\lambda>1$, we set
$$\lambda Q:= \{x\in E: \dist(x,Q)\leq (\lambda-1)\ell(Q)\}.$$
For $R\in\DD$, we denote $\DD(R) = \{Q\in\DD:Q\subset R\}.$

Let us remark that we will not need the ``small boundaries condition'' for the dyadic cubes (see \cite{DS2}).

In this paper, the usual cubes in $\R^d$ will be called ``true cubes'', to distinguish them from the ``false cubes'' in $\DD$.


\section{Elementary properties of the coefficients $\alpha(Q)$}\label{secalfa}

In the Introduction we defined the coefficients $\alpha(Q)$ for true cubes $Q\subset\R^d$. For $Q\in\DD$, the 
definition is the same. So given $Q\subset\R^d$, which may be either a true cube or a dyadic cube from $\DD$, we set
\begin{equation}\label{defalfa}
\alpha(Q) = \frac1{\ell(Q)^{n+1}}\,\inf_{c\geq0,L} \,\dist_{B_Q}(\mu,\,c\HH^n_{|L}),
\end{equation}
where the infimum is taken over all the constants $c\geq0$ and all the $n$-planes $L$, 
and $B_Q=B(z_Q,3\ell(Q))$. Observe that $Q\subset B_Q$. We denote by $c_Q$ and $L_Q$ a constant
and an $n$-plane that minimize $\dist_{B_Q}(\mu,\,c\HH^n_L)$ (it is easy to check that this minimum is attained). 
We also write 
$\LL_Q:=c_Q\HH^n_{|L_Q}$, so that
$$\alpha(Q) = \frac1{\ell(Q)^{n+1}}\,\dist_{B_Q}(\mu,\,c_Q\HH^n_{|L_Q})
= \frac1{\ell(Q)^{n+1}}\,\dist_{B_Q}(\mu,\,\LL_Q).$$
Let us remark that $c_Q$ and $L_Q$ (and so $\LL_Q$) may be not unique. Moreover, we may (and will) assume that 
$L_Q\cap B_Q\neq\varnothing$.



For simplicity, in the lemmas below we will work with dyadic cubes $Q\in D$. However, most of the results hold also for true cubes $Q$ such that $Q\cap \supp(\mu)\neq\varnothing$ and $\ell(Q)\lesssim\diam(\supp(\mu))$.

\begin{lemma} \label{lempr0}
For all $P,Q\in\DD$, the coefficients $\alpha(\cdot)$ satisfy the following properties:
\begin{itemize}
\item[(a)] $\alpha(Q) \lesssim 1$.
\item[(b)] If $P\subset Q$ and $\ell(P)\approx\ell(Q)$, then $\alpha(P) \lesssim\alpha(Q)$.
\item[(c)] If $\alpha(Q)\leq C_4$, with $C_4$ small enough, then $L_Q\cap B(z_Q,\diam(Q))\neq\varnothing$ and 
$c_Q\approx1$.
\end{itemize}
\end{lemma}

In a sense, (c) says that if $\alpha(Q)$ is small enough, then $L_Q$ is quite close to $Q$ (recall that by definition
we assumed $L_Q\cap  B_Q\neq\varnothing$ and that $B(z_Q,\diam(Q)) = \frac13 B_Q$).

\begin{proof}
The statement (a) is a direct consequence of the definitions. The property (b) follows from the fact that if $P\subset Q$, then $B_P\subset B_Q$ 
(recall that $\diam(R)\leq \ell(R)$ for all $R\in \DD$) and so
$\dist_{B_P}(\mu,\nu)\leq\dist_{B_Q}(\mu,\nu)$ for any given measure $\nu$.

Let us turn our attention to (c). To show that $L_Q\cap B(z_Q,\diam(Q))\neq\varnothing$ if $C_4$ is small enough, take a smooth function function 
$\vphi$ such that $\chi_{B(z_Q,\diam(Q)/10)} \leq \vphi\leq \chi_{B(z_Q,\diam(Q)/2)}$ with 
$\|\nabla\vphi\|_\infty\lesssim 1/\diam(Q)$. Then we have $\|\nabla(\vphi\,\dist(\cdot,L_Q))\|_\infty\lesssim 1$, and since
$\vphi\,\dist(\cdot,L_Q)$ vanishes on $L_Q$, we have
$$\biggl|\int\vphi(x)\dist(x,L_Q)\,d\mu(x)\biggr| \lesssim \alpha(Q)\ell(Q)^{n+1}.$$
On the other hand, 
\begin{align*}
\int\vphi(x)\dist(x,L_Q)\,d\mu(x) &\geq \dist(\supp(\vphi),\,L_Q) \int\vphi\,d\mu \\
& \gtrsim \dist(\supp(\vphi),\,L_Q) \,\mu(Q).
\end{align*}
If $\alpha(Q)$ is small enough we infer that
$\dist(\supp(\vphi),L_Q) \leq \diam(Q)/10$, and so $L_Q\cap B(z_Q,\diam(Q))\neq\varnothing.$

Let us check now that $c_Q\approx1$.
Let $\psi$ be a smooth function such that $\chi_{Q} \leq \psi\leq \chi_{B_Q}$ 
and $\|\nabla\psi\|_\infty\lesssim 1/\ell(Q)$. Then
$$\biggl|\int\psi\,d\mu - \int\psi\,d\LL_{Q}\biggr| \lesssim\alpha(Q)\mu(Q).$$
Thus, 
$$\int \psi\,d\mu - C\alpha(Q)\mu(Q) \leq c_Q\int\psi\,d\HH^n_{|L_{Q}} \leq \int \psi\,d\mu + C\alpha(Q)\mu(Q).$$
From the second inequality, we deduce easily that $c_Q\lesssim1$. From the first one, we see that if $\alpha(Q) \leq C_4$, where $C_4$ is 
small enough, then
$$ c_Q\int\psi\,d\HH^n_{|L_{Q}}  \geq \mu(Q) - C\alpha(Q)\mu(Q) \geq \frac12 \mu(Q),$$
which implies that $c_Q\gtrsim1$.
\end{proof}

Recall the definition of the bilateral $\beta_1$:
$$b\beta_1(Q) = \frac1{\ell(Q)^n}\,\inf_L \biggl[ \int_{2Q} \frac{\dist(y,L)}{\ell(Q)}\,d\mu(y) + 
\int_{L\cap B(z_Q,2\diam(Q))}\!\! \frac{\dist(x,E)}{\ell(Q)}\,d\HH^n_{|L}(x) \biggr],$$
where the infimum is taken over all the $n$-planes $L$ and $E=\supp(\mu)$. 
We have the following relationship between $\beta_1(Q)$, $b\beta_1(Q)$ and $\alpha(Q)$:

\begin{lemma} \label{lempr1}
For all $Q\in\DD$ we have
$$\beta_1(Q)\leq b\beta_1(Q)\lesssim\alpha(Q).$$
\end{lemma}

\begin{proof}
The first inequality is trivial. For the second one we may assume 
$\alpha(Q)\leq C_4$. 
Given an arbitrary $n$-plane $L$, we take
$$f(x) := \Bigl[\dist(x,L)- \dist(x,E)\Bigr]\vphi(x),$$
where $\vphi$ is a smooth function such that $\chi_{B(z_Q,2\diam(Q))} \leq \vphi\leq \chi_{B_Q}$ and $\|\nabla\vphi\|_\infty\lesssim 1/\ell(Q)$.
It is straightforward to check that $\|\nabla f\|_\infty\lesssim 1$.
As a consequence, for any $c\geq0$,
\begin{equation*}
\begin{split}
\dist_{B_Q}(\mu,\,c\HH^n_{|L}) & \gtrsim \biggl|\int f\,d\mu - c\int f\,d\HH^n_{|L}\biggr| \\
& = \biggl|\int \vphi(x)\dist(x,L)\,d\mu(x) + c\int \vphi(x)\dist(x,E)\,d\HH^n_{|L}(x)\biggr| \\
& \geq \min(1,c) b\beta_1(Q)\,\ell(Q)^{n+1}.
\end{split}
\end{equation*}
If we choose $L=L_Q$ and $c=c_Q$, we get
$$\dist_{B_Q}(\mu,\,\LL_{Q}) \gtrsim  b\beta_1(Q)\,\ell(Q)^{n+1},$$
since $c_Q\approx1$ (because $\alpha(Q)\leq C_4$),
and the lemma follows.
\end{proof}

\begin{remark} \label{rem1}
The calculations in the preceding lemma show that the following holds:
$$\frac1{\ell(Q)^n}\int_{B(z_Q,2\diam(Q))} \frac{\dist(y,L_Q)}{\ell(Q)}\,d\mu(y)  \lesssim \alpha(Q).$$
\end{remark}

\begin{lemma}\label{lempr2}
Let $P,Q\in\DD$ be dyadic cubes such that $P\subset Q$, with
$\eta\ell(Q)\leq \ell(P)\leq \ell(Q)$ for some fixed $\eta>0$. Then we have
\begin{equation}\label{dh1}
\dist_H\bigl(L_P\cap B_Q , \, L_Q\cap B_Q) \leq C(\eta)\alpha(Q)\ell(Q),
\end{equation}
where $\dist_H$ stands for the Hausdorff distance. Also,
\begin{equation} \label{dh2}
|c_P - c_Q| \leq C(\eta)\alpha(Q).
\end{equation}
\end{lemma}

\begin{proof}
All the constants in this proof (including the ones involved in the relationship ``$\lesssim$'') are allowed to
depend on $\eta$.

Clearly, we may assume that $\alpha(Q),\alpha(P) \leq C_4$. Otherwise, the statements in the lemma are trivial. First we prove \rf{dh1}.
Let $\vphi_P$ be a smooth function such that $\chi_{B(z_P,2\diam(P))} \leq \vphi_P\leq \chi_{B_P}$ and $\|\nabla\vphi_P\|_\infty\lesssim 1/\ell(P)$.
Then we have $\|\nabla(\vphi_P\, \dist(\cdot,L_Q))\|_\infty \lesssim 1$, and so
$$\biggl|\int\! \vphi_P(x)\dist(x,L_Q)\,d\mu(x) -\int\! \vphi_P(x)\dist(x,L_Q)\,d\LL_{P}(x) \biggr| 
\lesssim \alpha(P)\ell(P)^{n+1}.$$
Moreover, since $\vphi_P \dist(\cdot,L_Q)$ vanishes on $L_Q$, we have
$$\biggl|\int \vphi_P(x)\dist(x,L_Q)\,d\mu(x) \biggr| \lesssim \alpha(Q)\ell(Q)^{n+1}.$$
Taking into account that $\alpha(P)\lesssim\alpha(Q)$, we get
$$\int_{B(z_P,2\diam(P))}\dist(x,L_Q)\,d\LL_{P}(x) \lesssim \alpha(Q)\ell(Q)^{n+1}.$$
Using that $L_P\cap B(z_P,\diam(P))\neq\varnothing$, it can be shown that the above estimate implies \rf{dh1}.

To prove \rf{dh2}, take $\vphi_P$ as above. We have
$$\biggl|\int \vphi_P\,d\mu - c_P\int \vphi_P\,d\HH^n_{|L_P} \biggr| \lesssim \alpha(P)\mu(P)\lesssim \alpha(Q)\mu(Q).$$
Also,
$$\biggl|\int \vphi_P\,d\mu - c_Q\int \vphi_P\,d\HH^n_{|L_Q} \biggr| \lesssim \alpha(Q)\mu(Q).$$
Thus,
$$\biggl|c_P\int \vphi_P\,d\HH^n_{|L_P} - c_Q\int \vphi_P\,d\HH^n_{|L_Q} \biggr| \lesssim \alpha(Q)\mu(Q).$$
Then we have
\begin{equation}\label{dh3}
|c_Q-c_P| \int\vphi_P\,d\HH^n_{|L_P} \lesssim \alpha(Q)\mu(Q) + c_Q\biggl|\int \vphi_P\,d\HH^n_{|L_P} - \int \vphi_P\,d\HH^n_{|L_Q} \biggr|.
\end{equation}
Using \rf{dh1}, it can be shown that
$$\biggl|\int \vphi_P\,d\HH^n_{|L_P} - \int \vphi_P\,d\HH^n_{|L_Q} \biggr|\lesssim\alpha(Q)\mu(Q).$$
Since $c_Q\lesssim1$ and $\int\vphi_P\,d\HH^n_{|L_P} \approx\mu(Q)$, \rf{dh2} follows from \rf{dh3}.
\end{proof}


\section{The coefficients $\alpha$ on Lipschitz graphs and uniformly rectifiable sets}\label{secrectif}

To prove Theorem \ref{teolips} (and only for this theorem), for convenience we prefer to work with the family of the true dyadic cubes of $\R^d$, which we denote by $\DD_{\R^d}$.
Although, for a cube~$Q$ of this type the estimate $\mu(Q)\approx\ell(Q)^n$ may fail, for the cubes $Q\in
\DD_{\R^d}$ which intersect $\supp(\mu)$ we have $\mu(2Q)\approx\ell(Q)^n$. Recall that if $Q$ does not intersect $\supp(\mu)$, for convenience we set $\alpha(Q)=0$.

Given a true cube $Q$ with sides parallel to the axes, we denote by $I_Q$ the projection of~$Q$ onto the $n$ plane $\{(x_1,\ldots,x_d)\in\R^d:
x_{n+1}=\cdots=x_d=0\}$. Observe that $I_Q$ is an $n$-dimensional true cube.


\begin{proof}[\bf Proof of Theorem \ref{teolips}]
For $x\in\R^n$, we denote $\wt{A}(x):=(x,A(x))$, and we set
$$g(x) := \rho(\wt{A}(x))\,|J(\wt{A})(x)|,\qquad \mbox{for $x\in\R^n$},$$
where $J(\wt{A})(x)$ stands for the $n$-dimensional Jacobian of the map $x\mapsto \wt{A}(x)$.
Given a\linebreak cube~$Q\in\DD_{\R^d}$ which intersects $\supp(\mu)$, let $\wh{Q}$ be a cube with side 
length $16\ell(Q)$ such that $Q\in\DD_{\R^d}(\wh{Q})$ and $B_Q\subset\wh{Q}$. We will show that
\begin{equation} \label{sh1}
\alpha(Q) \lesssim \beta_1(2Q) + \sum_{I\in\DD_{\R^n}(I_\wh{Q})}\frac{\ell(I)^{1+n/2}}{\ell(Q)^{1+n}}\,\|\D_{I} g\|_2,
\end{equation}
where $\DD_{\R^n}(I_\wh{Q})$ stands for the collection of dyadic cubes from $\DD_{\R^n}$ which are contained in $I_\wh{Q}$, the projection of
$\wh{Q}$ onto the $n$-plane $\{x\in\R^d:\,x_{n+1}=\cdots=x_d=0\}$.
Notice that in this inequality $Q$ is a cube in $\R^d$ and the $I$'s are cubes in $\R^n$. Also, the $L^2$~norm
is taken with respect to the $n$-dimensional Lebesgue measure on $\R^n$.
Let us see that the theorem follows from the preceding estimate. Indeed, we derive
\begin{align*}
\sum_{Q\in \DD_{\R^d}(R)}\alpha(Q)^2\,\mu(Q) &\lesssim  \sum_{Q\in \DD_{\R^d}(R)}\beta_1(2Q)^2\mu(Q) \\
 & \quad+
\sum_{Q\in\DD_{\R^d}(R)}\biggl(\sum_{I\in\DD_{\R^n}(I_\wh{Q})}\frac{\ell(I)^{1+n/2}}{\ell(Q)}\,\|\D_{I} g\|_2\biggr)^2
 \frac1{\mu(Q)}.
\end{align*}
The first sum on the right hand side is bounded above by $C\ell(R)^n$, by the results of David and Semmes.
To deal with the last sum, which we denote by $S_2$, we apply Cauchy-Schwartz:
$$S_2\leq 
\sum_{Q\in\DD_{\R^d}(R)}\biggl(\sum_{I\in\DD_{\R^n}(I_\wh{Q})}\frac{\ell(I)}{\ell(Q)}\,\|\D_{I} g\|_2^2\biggr)
\biggl(\sum_{I\in\DD_{\R^n}(I_\wh{Q})}\frac{\ell(I)}{\ell(Q)}\,\ell(I)^n\biggr)\frac1{\mu(Q)}.$$
Since $\sum_{I\in\DD_{\R^n}(I_\wh{Q})}\frac{\ell(I)}{\ell(Q)}\,\ell(I)^n\lesssim\mu(Q)$, we get
\begin{align*} S_2 &\lesssim 
\sum_{Q\in\DD_{\R^d}(R)}\sum_{I\in\DD_{\R^n}(I_\wh{Q})}\frac{\ell(I)}{\ell(Q)}\,\|\D_{I} g\|_2^2 \\ &\leq 
\sum_{I\in\DD_{\R^n}(I_\wh{R})} \|\D_{I} g\|_2^2 \sum_{Q\in\DD_{\R^d}(R):I_{\wh{Q}}\supset I}\frac{\ell(I)}{\ell(Q)}
\lesssim \sum_{I\in\DD_{\R^n}(I_\wh{R})} \|\D_{I} g\|_2^2.
\end{align*}
Since $g$ is a bounded function, we deduce that $S_2\lesssim\ell(R)^n$.

It remains to show that \rf{sh1} holds. For any $x\in I_{4Q}\subset I_{\wh{Q}}$, we have
$$g(x) = \sum_{I\in\DD_{\R^n}(I_{\wh{Q}})}\Delta_I g(x) + g_{I_{\wh{Q}}}.$$
To estimate $\alpha(Q)$, we choose an $n$-plane $L=\{(x,y)\in \R^n\times \R^{d-n}:\,y=a(x)\}$ (where $a(x)$ is an
appropriate affine map) which minimizes $\beta_1(3Q)$ and we set 
$c_Q:=  g_{I_{\wh{Q}}}/|J(\wt{a})|$ (notice that the Jacobian of the map $x\mapsto \wt{a}(x):=(x,a(x))$
is constant and bounded from below). 
Given a Lipschitz function $f$ supported in $B_Q$ with Lipschitz constant $\leq1$,
we have
\begin{align*}
 &\quad\; \biggl| \int f(z)\,d\mu(z) - c_Q\int f(z)\,d\HH^n_{|L}\biggr| \\ &= 
\biggl|\int f(\wt{A}(x))\,g(x) dx - c_Q\int f(\wt{a}(x)) |J(\wt{a})| \,dx\biggr| \\
&\leq  \biggl|\int f(\wt{A}(x))\,g_{I_{\wh{Q}}} dx - \!\int f(\wt{a}(x)) g_{I_{\wh{Q}}} \,dx\biggr| + \!\!\!
 \sum_{I\in\DD_{\R^n}(I_{\wh{Q}})}\biggl| \int f(\wt{A}(x)) \Delta_Ig(x)  \,dx\biggr|\\
& =: T_1 + T_2.
\end{align*}
To deal with $T_1$ we take into account that $g_{I_{\wh{Q}}}$ is bounded (since
$g$ is a bounded function) and that $f$ is $1$-Lipschitz and $\supp(f)\subset B_Q\subset 6Q$:
$$T_1 \lesssim \int |f(\wt{A}(x)) - f(\wt{a}(x))| \,dx \leq
\int_{6I_Q} |A(x)-a(x)| \,dx \lesssim \beta_1(2Q)\ell(Q)^{1+n}.$$
For $T_2$ we use the fact that $\Delta_I g$ has mean value zero and $f$ and $\wt{A}$ are Lipschitz maps:
\begin{align*}
T_2 &=  \sum_{I\in\DD_{\R^n}(I_{\wh{Q}})}\biggl| \int_{I} (f(\wt{A}(x))- f(\wt{A}(x_I)) \Delta_Ig(x)  \,dx\biggr|\\
&\lesssim \sum_{I\in\DD_{\R^n}(I_{\wh{Q}})} \ell(I)\|\Delta_I g\|_1 \leq 
\sum_{I\in\DD_{\R^n}(I_{\wh{Q}})} \ell(I)^{1+n/2}\|\Delta_I g\|_2.
\end{align*}
From the estimates of $T_1$ and $T_2$ and the definition of $\alpha(Q)$, we get \rf{sh1}.
\end{proof}

\begin{remark}
Let $\Gamma$ be an 
$n$-dimensional Lipschitz graph with compact support in $\R^{n+1}$. That is, 
$\Gamma:=\{(x,y)\in \R^n\times\R:\,y=A(x)\}$. Set $\mu= \HH^n_{|\Gamma}$.
By the calculations in the preceding theorem, we have
$$\sum_{Q\in \DD_{\R^{n+1}}}\alpha(Q)^2\,\mu(Q) \lesssim  \sum_{Q\in \DD_{\R^n}}\beta_1(2Q)^2\mu(Q) + \sum_{I\in\DD_{\R^n}}\|\Delta_I g\|_2^2,$$
where 
$$g(x) = |J(\wt{A})(x)| = (1+|\nabla A(x)|^2)^{1/2}.$$
Notice that 
$$|g(x)-1| = \frac{|\nabla A(x)|^2}{1+(1+|\nabla A(x)|^2)^{1/2}}\leq\min(|\nabla A(x)|,\,|\nabla A(x)|^2).$$
So we have
$$\sum_{I\in\DD_{\R^n}}\|\Delta_I g\|_2^2\leq C \int_{\R^n}  |g(x) -1|^2 dx \leq \|\nabla A\|_2^2.$$
On the other hand, by \cite[Theorem 6]{Dorronsoro}, we also have
$$\sum_{Q\in \DD_{\R^n}}\beta_1(2Q)^2\mu(Q)\approx \sum_{Q\in \DD_{\R^n}}\beta_1(Q)^2\mu(Q)\approx \|\nabla A\|_2^2.$$
Recalling that $\beta_1(Q)\lesssim \alpha(Q)$ for any cube $Q$, we get
$$\sum_{Q\in \DD_{\R^{n+1}}}\alpha(Q)^2\,\mu(Q) \approx \|\nabla A\|_2^2,$$
with constants depending on $C_1$ in Theorem \ref{teolips}.
\end{remark}

\vvv

\begin{proof}[\bf Proof of Theorem \ref{teounif}]
It is clear that (b) $\Rightarrow$ (c). On the other hand,
we have shown in Lemma \ref{lempr1} that $b\beta_1(Q)\lesssim \alpha(Q)$ for any cube $Q\in\DD$.
Thus, if the coefficients $\alpha$ satisfy the packing condition
$$
\sum_{Q\in \DD(R)}\alpha(Q)^2\,\mu(Q) \lesssim \mu(R) \qquad \mbox{for all $R\in\DD$},
$$
then an analogous inequality holds if we replace $\alpha$ by $b\beta_1$ or by $\beta_1$. As a consequence,
$\mu$ is uniformly rectifiable in this case, by the results of David and Semmes in \cite{DS1}. Thus, (b)  $\Rightarrow$ (a) in Theorem \ref{teounif}.

Analogously, if (c) hods, then for all $\ve>0$, there exists some constant $C(\ve)$ such that the collection ${\mathcal B}_\ve'$ of those cubes $Q\in\DD$ such that $b\beta_1(Q)>\ve$ satisfies
$$\sum_{Q\in \mathcal{B}_\ve': Q\subset R} \mu(Q)\leq C(\ve)\mu(R)$$
for any cube $R\in\DD$. As a consequence, $E=\supp(\mu)$ satisfies the so called bilateral weak geometric lemma and by \cite[Theorem 2.4]{DS2}
$\mu$ is uniformly rectifiable. That is, (c) $\Rightarrow$ (a).

The proof that (a) $\Rightarrow$ (b) is more technical. We give only some hints: this follows by using the fact that uniformly 
rectifiable
sets admit corona decompositions (see \cite{DS1} or \cite{DS2} for the precise definition). Then the arguments are similar to the ones
in \cite[Section 15]{DS1}, where it is shown that the existence of a corona decomposition implies that the $\beta_1(Q)$'s satisfy a 
packing condition like the one in \rf{pack0}. The idea consists of constructing a partition of $\DD$ into sets (that we call trees) such
that on each tree $\mu$ is well approximated in some precise sense by $n$-dimensional Hausdorff measure on a Lipschitz graph. Then one
uses the fact that the $\alpha$'s satisfy a Carleson packing condition on Lipschitz graphs (because of Theorem \ref{teolips}), and one argues by approximation on each tree.
\end{proof}


\section{Estimates for Calder\'on-Zygmund operators in terms of the coefficients~$\alpha$}\label{secz1}

\subsection{The \boldmath$L^2(\mu)$ boundedness of \boldmath$T_\mu$}\label{sub1}

When $\mu$ is uniformly rectifiable, the fact that $T_\mu$ is bounded in $L^2(\mu)$ is a consequence of inequality \rf{pral11} and the $T(1)$ theorem. Indeed, suppose first that $\mu$ is
supported on a Lipschitz graph. Then for any dyadic cube $R$,
$$\sum_{Q\in\DD(R)} \alpha(Q)^2\mu(Q)\lesssim \mu(R).$$
Thus if we apply inequality \rf{pral11} to $\mu_{|R}$ (which is itself $AD$ regular), then we deduce that
$\|T_*(\chi_R \mu)\|_{L^2(\mu)}\lesssim \mu(R)^{1/2}$, and so $T_\mu$ is bounded in $L^2(\mu)$ by the $T(1)$ theorem.

If $\mu$ is uniformly rectifiable but not supported on a Lipschitz graph, then $T_\mu$ is also bounded in $L^2(\mu)$ because of the ``big pieces
functor'' (see Proposition I.1.28 of \cite{DS2}). An alternative argument consists of using Theorem \ref{teounif}, which implies that the 
coefficients $\alpha(Q)$ satisfy the packing condition above, and then the same proof given for $\mu$ supported on a Lipschitz graph works in
this case.

Section \ref{sec6} and the rest of the present section are devoted to the proof of inequalities \rf{pral11} and
\rf{pral110} in Theorem \ref{teomain}.


\subsection{Decomposition of \boldmath$T\mu$ with respect to \boldmath$\DD$}
\label{subdec}
As explained in the Introduction, to prove \rf{pral11} and \rf{pral110} we will
decompose~$T\mu$ using the dyadic lattice $\DD$ associated to $\mu$.
Let $\psi$ be a non-increasing radial $\CC^\infty$ function such
that $\chi_{B(0,1/2)}\leq \psi\leq \chi_{B(0,2)}$. For each
$j\in\Z$, set $\psi_j(z) := \psi(2^jz)$ and 
$\vphi_j:=\psi_{j+3}-\psi_{j+4}$ (recall that in the Introduction we set 
$\vphi_j:=\psi_{j}-\psi_{j+1}$; for simplicity in some calculations below, we prefer the
choice $\vphi_j:=\psi_{j+3}-\psi_{j+4}$),  so
 that each function $\vphi_j$ is
non negative and supported in $A(0, 2^{-j-5}, 2^{-j-2}))$, and moreover we have
$$\sum_{j\in\Z}\vphi_j(x) = 1$$
for any $x\in\R^d\setminus \{0\}$. For
each $j\in Z$ we denote $K_j(x) = \vphi_j(x)\,K(x)$ and
\begin{equation}\label{deftn}
T_j \mu(x) = \int K_j(x-y)\,d\mu(y).
\end{equation}
For each $Q\in\DD$, we set
$$T_Q \mu := \chi_Q T_{J(Q)}\mu.$$
Recall that $J(Q)$ stands for the integer such that $Q\in\DD_j$. Formally we have
$$
 T\mu = \sum_{m\in\Z} T_m\mu \,=\, \sum_{m\in\Z}\,\,
\sum_{Q\in\DD_m} T_Q\mu.$$
This decomposition of $T\mu$ is inspired in part by \cite{Semmes}. See also \cite{Tolsa-pubmat2}, \cite{MT} 
and \cite{Tolsa-bilip} for some
related techniques.

Let us denote
$$T_{(m)} \mu = \sum_{j:j\leq m} T_j\mu.$$
To prove the estimates \rf{pral11} and \rf{pral110} in Theorem \ref{teomain}, we will show that
\begin{equation}\label{pral12}
\|T_{(m)}\mu\|_{L^2(\mu)}^2\lesssim \sum_{Q\in\DD} \alpha(Q)^2\mu(Q),
\end{equation}
uniformly on $m\in\Z$. By the following ``Cotlar type'' inequality:
\begin{equation}\label{eqcotlar}
\|T_*\mu\|_{L^2(\mu)} \lesssim \limsup_{m\to\infty} \|T_{(m)}\mu\|_{L^2(\mu)} + \mu(\C), 
\end{equation}
this implies \rf{pral11}. 
The proof of \rf{eqcotlar} follows by arguments analogous to the ones of the usual Cotlar inequality although, in our particular case, it is not necessary to use the $L^2(\mu)$ boundedness of $T_\mu$ (which we are not assuming) to prove it,
because of the antisymmetry of $T_\mu$. See for instance \cite[Lemma 5.1]{Volberg} for a similar estimate.

The existence of $\pv T\mu(x)$ for $\mu$-a.e. $x$ under the assumption
$$
\sum_{Q\in \DD}\alpha(Q)^2 \mu(Q)<\infty
$$
is a consequence of the fact that this implies that $\supp(\mu)$ must be $n$-rectifiable 
(although perhaps not uniformly rectifiable), and so $\mu$
is supported on a countable union of $n$-dimensional Lipschitz graphs. Since $T_{\HH^n|\Gamma}$ is
bounded on $L^2(\HH^n_{|\Gamma})$ for any $n$-dimensional Lipschitz graph~$\Gamma$ 
(by Subsection \ref{sub1}), 
it follows that $\pv T\mu(x)$ exists $\mu$-a.e. The arguments are similar to the ones in 
\cite[Chapter 20]{Mattila-llibre}. On the other hand, $T_{(m)}\mu(x)$ can be written as a convex combination of~$T_\ve\mu(x)$,
with $0<\ve\leq 2^{-m-3}$. Indeed, if we set $g(|x-y|)= 1- \psi_{m+4}(x-y)$, we have 
$$T_{(m)}\mu(x)  = \sum_{j:j\leq m}T_j\mu(x) = \int g(|x-y|)\,K(x-y)\,d\mu(y),$$
and then by Fubini it is easy to check that
$$\int g(|x-y|)\,K(x-y)\,d\mu(y) = \int_0^\infty g'(\ve)\,T_\ve\mu(x)\,d\ve=
\int_0^{2^{-m-3}} g'(\ve)\,T_\ve\mu(x)\,d\ve,$$
since $g(\ve)= 1$ for $\ve\geq2^{-m-3}$. Then,
 it turns out that whenever $\pv T\mu(x)$ exists we have
$$\lim_{m\to\infty}T_{(m)}\mu(x) = \pv T\mu(x).$$
Since $T_* \mu\in L^2(\mu)$, by dominated convergence we derive
$$\|\pv T\mu\|_{L^2(\mu)}^2 = \lim_{m\to\infty} \|T_{(m)}\mu\|_{L^2(\mu)}^2 \lesssim
 \sum_{Q\in\DD}\alpha(Q)^2\mu(Q).$$

So it only remains to prove \rf{pral12}. To this end we set
$$\|T_{(m)}\mu\|_{L^2(\mu)}^2 = \sum_{j:j\leq m} \|T_j\mu\|_{L^2(\mu)}^2 + 2\sum_{j,k:j<k\leq m} \langle T_j\mu,T_k\mu\rangle.$$
We will see that the first sum on the right side is easy to estimate, while the last sum will require some harder 
work. To estimate this  last sum we will show that the oddness of $K(\cdot)$ introduces
some quasiorthogonality among the different functions $T_j\mu$, $j\in\Z$. 


\subsection{Estimates for \boldmath$T_Q \mu$ in terms of the coefficients \boldmath$\alpha$}

Given $Q\in\DD$, we denote $${\mathcal A}^2(Q):=
\frac1{\mu(Q)}\sum_{P\in\DD(Q)} \alpha(P)^2\,\frac{\ell(P)}{\ell(Q)}\,\mu(P).$$

\begin{lemma}\label{lempr3}
Given $Q\in\DD_m$, we have
\begin{itemize}
\item[(a)] For any $x\in L_Q\cap B(z_Q,2\diam(Q))$, $|T_m\mu(x)|\lesssim\alpha(Q)$.
\item[(b)] $\ds \int |T_Q\mu|\,d\mu \lesssim\alpha(Q)\mu(Q)$.
\item[(c)] $\ds \int |T_Q\mu|^2\,d\mu \lesssim {\mathcal A}^2(Q)\mu(Q).$
\end{itemize}
\end{lemma}

\begin{proof}
Let us prove (a). Notice that if $x\in B(z_Q,2\diam(Q))$, then $\supp \bigl[K_m(x-\cdot)\bigr]\subset B_Q$. Moreover, $\|\nabla 
K_m(x-\cdot)\|_\infty\lesssim
\ell(Q)^{-n-1}$. Given an arbitrary $n$-plane $L$ and $c\geq0$, if $x\in L$ we have
$$\biggl|\int K_m(x-y)\,d\mu(y) - c\int K_m(x-y)\,d\HH^n_{|L}(y)\biggr| \lesssim \frac1{\ell(Q)^{n+1}}\,\dist_{B_Q}(\mu,c\HH^n_{|L}).$$
Since $K_m(\cdot)$ is odd, the second integral vanishes, and so
$$|T_m\mu(x)| \lesssim\frac1{\ell(Q)^{n+1}}\,\dist_{B_Q}(\mu,c\HH^n_{|L}).$$
If we choose $L=L_Q$ and $c=c_Q$, the statement (a) follows.

Now we will prove (b) and (c) simultaneously.
Given $x\in Q$, let $x'$ be the orthogonal projection of $x$ onto $L_Q$. Then we have
$$|T_m\mu(x) - T_m\mu(x')|\lesssim \frac{|x-x'|}{\ell(Q)^{n+1}}\,\mu(Q) \lesssim\frac{\dist(x,L_Q)}{\ell(Q)}.$$
We may assume that $x'\in B(z_Q,2\diam(Q))$ because otherwise $\dist(x,L_Q)\geq\ell(Q)/2$ and this would mean that 
$\AZ(Q)\geq\alpha(Q)\gtrsim1$ (see Remark \ref{rem1}), and then (b) and (c) would 
be trivial. So, using (a), we have
$$|T_m\mu(x)| \lesssim \frac{\dist(x,L_Q)}{\ell(Q)} + |T_m\mu(x')| \lesssim \frac{\dist(x,L_Q)}{\ell(Q)} + \alpha(Q).$$
Thus, for $p\geq1$,
$$\int|T_Q\mu|^p\,d\mu \lesssim \int_Q\biggl(\frac{\dist(x,L_Q)}{\ell(Q)}\biggr)^p\,d\mu(x) + \alpha(Q)^p\mu(Q).$$
For $p=1$ the first integral on the right side is bounded above by $C\alpha(Q)\mu(Q)$, by Remark~\ref{rem1}, and (b) follows.
On the other hand, if we choose $p=2$, (c) is a consequence of next lemma.
\end{proof}

\begin{lemma} \label{lemdif1}
Let $x\in\supp(\mu)$ and $Q\in\DD$ such that $x\in Q$. We have
\begin{equation}\label{claim1}
\dist(x,L_Q) \lesssim \sum_{P\in\DD:P\subset Q,\,x\in P}\alpha(P)\,\ell(P),
\end{equation}
and
\begin{equation}\label{claim11}
\int_Q\biggl(\frac{\dist(x,L_Q)}{\ell(Q)}\biggr)^2\,d\mu(x) \lesssim {\mathcal A}^2(Q)\mu(Q).
\end{equation}
\end{lemma}

\begin{proof}
Let us see how \rf{claim11} follows from \rf{claim1}. By Cauchy-Schwartz we have
\begin{align*}
\biggl(\frac{\dist(x,L_Q)}{\ell(Q)}\biggr)^2 &\lesssim \biggl(\sum_{P\in\DD:P\subset Q,\,x\in P}\alpha(P)^2\,
\frac{\ell(P)}{\ell(Q)}\biggr) 
\biggl(\sum_{P\in\DD:P\subset Q,\,x\in P}
\frac{\ell(P)}{\ell(Q)}\biggr) \\
& \lesssim \sum_{P\in\DD:P\subset Q,\,x\in P}\alpha(P)^2\,
\frac{\ell(P)}{\ell(Q)},
\end{align*}
and then \rf{claim11} follows by integration on $Q$.

To prove \rf{claim1}, let $n_0\geq1$ be some integer to be fixed below, and consider the sequence of dyadic 
cubes $Q=Q_0\supset Q_1\supset Q_2\ldots$ such that $x\in Q_m$ 
for each $m\geq1$ and $\ell(Q_m)=2^{-mn_0}\ell(Q)$. 
Let $\ve_0$ be some (small) constant that will be fixed below too.
Let $N\geq0$ be the least integer such that $\alpha(Q_N)\geq\ve_0$. If $N$ does not exist because 
$\alpha(Q_m)<\ve_0$ for all~$m\geq0$, we let $N$ be an arbitrary positive integer. Let $a_N$ be 
any point from $Q_N$ and for $m=N-1,\,N-2,\ldots,0$ let $a_m$ be the orthogonal projection of $a_{m+1}$
onto $L_{Q_m}$. Then we have
\begin{equation} \label{cll0}
\dist(a_N,L_Q) \leq \dist(a_N,L_{Q_N}) + \sum_{m=0}^{N-1} \dist(a_m,a_{m+1}).
\end{equation}
Our next objective consists in showing that 
\begin{equation}\label{cll2}
|a_m - a_{m+1}|\lesssim \alpha(Q_m)\ell(Q_m)\quad \mbox{ for $m=0,\,1,\ldots,N-1$.}
\end{equation}
Let us see first that \rf{claim1} follows from this estimate. Indeed, from \rf{cll0} and \rf{cll2} we infer
$$\dist(a_N,L_Q) \lesssim \dist(a_N,L_{Q_N}) + \sum_{m=0}^{N-1} \alpha(Q_m)\ell(Q_m).$$
If $\alpha(Q_m)<\ve_0$ for all $m\geq0$ (in this case $N$ was chosen as an arbitrary positive integer), we let $N\to\infty$ in the preceding
inequality, and then $a_N\to x$ and $\dist(a_N,L_{Q_N})\to0$, and so \rf{claim1} follows. If $\alpha(Q_N)\geq\ve_0$, then 
$$|x-a_N| + \dist(a_N,L_{Q_N})\lesssim\ell(Q_N) \leq \ve_0^{-1}\alpha(Q_N)\ell(Q_N).$$
Thus, by \rf{cll0} and \rf{cll2},
\begin{align*}
\dist(x,L_Q) & \leq |x-a_N| + \dist(a_N,L_{Q_N}) + \sum_{m=0}^{N-1} \dist(a_m,a_{m+1})
 \lesssim \sum_{m=0}^{N} \alpha(Q_m)\ell(Q_m).
\end{align*}

To prove \rf{cll2} we wish to apply Lemma \ref{lempr2}. Then we 
need to show first that $a_m\in B_{Q_m}$ for $m=N,\,N-1,\ldots,1$. 
We argue by backward induction. Indeed, for $m=N$, this holds by the definition of $a_N$. Assume now that $a_{m+1}\in B_{Q_{m+1}}$
and let us see that $a_m\in B_{Q_m}$. Remember that for $m=N-1,\,N-2,\ldots,1$, we have $\alpha(Q_m)\leq\ve_0$. By the AD regularity of~$\mu$, 
all points $y\in Q_{m+1}\subset Q_m$ satisfy
$$\dist(y,L_{Q_m})\leq C(\ve_0)\ell(Q_m)\leq\diam(Q_{m+1})/2,$$
assuming that $\ve_0$ has been chosen small enough (depending also on choice of $n_0$). So 
we infer that $L_{Q_m}\cap B_{Q_{m+1}}\neq\varnothing$. Recall 
that, by the induction hypothesis, $a_{m+1}\in B_{Q_{m+1}}$. If $n_0$ has been chosen big enough we deduce that 
\begin{equation} \label{auxx1}
|a_{m+1} - a_m| = \dist(a_{m+1},L_{Q_m})\leq\diam(B_{Q_{m+1}}) \leq\diam(Q_m)/4.
\end{equation}
It is straightforward to check that $B_{Q_{m+1}}\subset B(z_{Q_m},2\diam(Q_m))$ for $n_0$ big enough. Thus we have $a_{m+1}\in 
B(z_{Q_m},2\diam(Q_m))$.  This fact and \rf{auxx1} imply that $a_m\in B_{Q_m}$.

The estimate \rf{cll2} follows now easily from Lemma \ref{lempr2} using the fact that $a_m,a_{m+1}\in B_{Q_m}$:
$$|a_m- a_{m+1}| \leq \dist_H (L_{Q_m}\cap B_{Q_m},\, L_{Q_{m+1}}\cap B_{Q_m}) \lesssim\alpha(Q_m)\,\ell(Q_m).$$
\end{proof}

\begin{remark}\label{remfa}
Almost the same arguments used to prove \rf{claim1} show that if $S,Q\in\DD$ are cubes such that $S\subset Q$
and $x\in\supp(\mu)\cap 2S$, then we have
\begin{equation}\label{claim12}
\dist(x,L_Q) \lesssim \dist(x,L_S) + \sum_{P\in\DD:S\subset P\subset Q}\alpha(P)\,\ell(P).
\end{equation}
\end{remark}


\subsection{Estimate of \boldmath$\sum_{j\in\Z}\|T_j\mu\|_{L^2(\mu)}^2$} 

The following lemma is an easy consequence of (c) in Lemma \ref{lempr3}.

\begin{lemma} \label{lemtq2}
For every $R\in\DD$, we have
$$\sum_{Q\in\DD(R)}\|T_Q\mu\|_{L^2(\mu)}^2 \lesssim\sum_{Q\in\DD(R)}\alpha(Q)^2\mu(Q).$$
\end{lemma}

\begin{proof}
By (c) in Lemma \ref{lempr3} we have
\begin{align*}
\sum_{Q\in\DD(R)}\|T_Q\mu\|_{L^2(\mu)}^2 & \lesssim \sum_{Q\in\DD:Q\subset R} {\mathcal A}^2(Q)\mu(Q) \\ & =
\sum_{Q\in\DD:Q\subset R} \sum_{P\in\DD:P\subset Q} \alpha(P)^2\,\frac{\ell(P)}{\ell(Q)}\,\mu(P)\\
& = \sum_{P\in\DD:P\subset R} \alpha(P)^2 \mu(P)\sum_{Q\in\DD:P\subset Q\subset R} \frac{\ell(P)}{\ell(Q)}\\
& \lesssim \sum_{P\in\DD:P\subset R} \alpha(P)^2 \mu(P).
\end{align*}
\end{proof}

So we have
\begin{equation}\label{sumj0}
\sum_{j\in\Z}\|T_j\mu\|_{L^2(\mu)}^2 \lesssim \sum_{Q\in\DD}\alpha(Q)^2\mu(Q).
\end{equation}

\begin{remark}
In \cite{DS1} it is shown that the following condition is necessary and sufficient for~$\mu$ to be uniformly
rectifiable: for each $\CC^\infty$, compactly supported, odd function $\psi:\R^d\to\R$, there is a $C>0$
such that for any cube $R\in\DD$,
\begin{equation}\label{sumj1}
\sum_{j\geq J(R)} \int_{x\in R} \biggl| \int \psi_j(x-y)\,d\mu(y)\biggr|^2 d\mu(x) \leq C\mu(R),
\end{equation}
where $\psi_j(x) = 2^{jn}\psi(2^{j}x)$.

For any $n$ dimensional AD regular measure $\mu$, it is easy to check that
\begin{equation}\label{sumj2}
\sum_{j\in\Z} \int \biggl| \int \psi_j(x-y)\,d\mu(y)\biggr|^2 d\mu(x) \lesssim
\sum_{Q\in\DD} \alpha(Q)^2\mu(Q).
\end{equation}
The arguments are very similar to the ones we used to obtain \rf{sumj0}. The role of $T_j\mu$ in \rf{sumj0} is
played now $\int \psi_j(x-y)\,d\mu(y)$. 
As a consequence, if $\mu$ is uniformly rectifiable, then \rf{sumj1} can be deduced from \rf{sumj2} applied
to $\mu_{|R}$. This way of proving that uniformly rectifiable measures satisfy \rf{sumj1} is very different from the
one in \cite{DS1}.
\end{remark}


\section{Estimate of $\sum_{j,k:k>j}\langle
T_j\mu,T_k\mu\rangle$ in terms of the $\alpha$'s} \label{sec6}

In this section we will show that
$$\sum_{j,k:k>j}\bigl|\langle T_j\mu,T_k\mu\rangle\bigr| \lesssim \sum_{Q\in\DD}\alpha(Q)^2\mu(Q).$$
The key idea consists in using quasiorthogonality.
This will finish the proof of Theorem \ref{teomain}.

Given $k>j$ fixed, let $m=[(j+k)/2]$, where $[\cdot]$ stands for the integer part. We write
$$\langle T_j\mu,T_k\mu\rangle = \sum_{S\in \DD_m} \langle \varphi_S T_j\mu,T_k\mu\rangle,$$
where $\{\varphi_S\}_{S\in\DD_m}$ is a family of $\CC^\infty$ functions such that each $\vphi_S$ satisfies
 $\supp(\varphi_S) \subset U_{\ell(S)/10}(S)$ (where $U_{\ell(S)/10}(S)$ stands for the $(\ell(S)/10)$-neighborhood of $S$) and 
$\|\nabla\varphi_S\|_\infty \leq C/\ell(S)$, and moreover $\sum_{S\in\DD_m}\varphi_S=1$ on $E$.
Let $x_S$ be the orthogonal projection of the center of $S$, $z_S$, onto $L_S$. We set
\begin{equation}\label{eqAB}
\begin{split}
\langle T_j\mu,T_k\mu\rangle & = \sum_{S\in \DD_m} \bigl\langle \varphi_S \bigl(T_j\mu - T_j\mu(x_S)\bigr),T_k\mu\bigr\rangle
+ \sum_{S\in \DD_m}  T_j\mu(x_S) \langle \varphi_S,T_k\mu\rangle \\
 &=: A_{j,k} + B_{j,k}.
 \end{split}
\end{equation}


\subsection{Estimates for \boldmath$A_{j,k}$ in \rf{eqAB}} \label{subsecajk}

We write $A_{j,k}$ as follows:
$$
A_{j,k} = \sum_{R\in\DD_j}\,\,
\sum_{S\in \DD_m:S\subset R} \bigl\langle \varphi_S \bigl(T_j\mu - T_j\mu(x_S)\bigr),T_k\mu\bigr\rangle.$$
Thus,
$$\sum_{j,k:k>j} A_{j,k} = 
 \sum_{R\in\DD}\,\sum_{k>J(R)}\,
\sum_{S\in \DD_m:S\subset R}\!\!\!\!  \bigl\langle \varphi_S \bigl(T_{J(R)}\mu - 
T_{J(R)}\mu(x_S)\bigr),T_k\mu\bigr\rangle =:  \sum_{R\in\DD} A_R,$$ 
where $J(R)$ stands for the generation of $R$.

We will need the following lemma.

\begin{lemma}\label{lempr4}
 Given $Q\in \DD_m$ and $x,y\in B(z_Q,2\diam(Q))$, we have
\begin{equation}\label{nab}
\begin{split}
|T_m\mu(x) - T_m\mu(y)| &\lesssim  \frac{\alpha(Q)\ell(Q)+
\dist(x,L_Q) + \dist(y,L_Q)}{\ell(Q)^2}\,|x-y| \\ 
&\quad+
\frac{|\Pi_{L_Q^\bot}(x-y)|}{\ell(Q)},
\end{split}
\end{equation}
where $\Pi_{L_Q^\bot}$ denotes the orthogonal projection on the the subspace orthogonal to $L_Q$.
\end{lemma}

Let us remark that in the proof of the preceding lemma we will use the assumption
\begin{equation} \label{smoot1}
|\nabla^2 K(x)| \leq \frac{C}{|x|^{n+2}} \qquad \forall x\in\R^d\setminus \{0\}.
\end{equation}
This is the only place in this paper where it is used.

\begin{proof}
Let $u$ be a unit vector parallel to $L_Q$. First we will show that for any $x\in B(z_Q,\frac52 \diam(Q)) $
\begin{equation}\label{derivdic}
|\nabla_u T_m\mu(x)| \lesssim\frac{\alpha(Q)}{\ell(Q)} + \frac{\dist(x,L_Q)}{\ell(Q)^2},
\end{equation}
where $\nabla_u$ stands for the directional derivative in the direction of $u$. Indeed, for any $x\in B(z_Q,\frac52 \diam(Q))$,
$$\nabla_u T_m\mu(x) = \int\nabla_u K_m(x-y)\,d\mu(y).$$
By the assumption \rf{smoot1}, we have
$$\biggl| \int\nabla_u K_m(x-y)\,d\mu(y) - \int\nabla_u K_m(x-y)\,d\LL_{Q}(y)\biggr|\lesssim\frac{\alpha(Q)}{\ell(Q)}.$$
For $x\in L_Q\cap B(z_Q,\frac52 \diam(Q))$, notice that the second integral on the left hand side above vanishes because 
$T_m\LL_{Q}$ vanishes identically on $L_Q$, and so $\nabla_u T_m\LL_{Q}(x)=0$. Therefore, 
\begin{equation} \label{eqx1}
|\nabla_u T_m\mu(x)| \lesssim\frac{\alpha(Q)}{\ell(Q)} \quad\mbox{ if $x\in B(z_Q,\frac52 \diam(Q))\cap L_Q$.}
\end{equation}
Consider now $x\in B(z_Q,2 \diam(Q))$ and let $x'$ be the orthogonal projection of $x$ onto $L_Q$. We may assume that 
$x'\in B(z_Q,\frac52 \diam(Q))$ because
otherwise $\dist(x,L_Q)\gtrsim\ell(Q)$ and then \rf{derivdic} is trivial in this case. 
Thus, from \rf{eqx1}, if $x\in B(z_Q,2 \diam(Q))$ and $x'\in B(z_Q,\frac52 \diam(Q))$ we get
$$|\nabla_u T_m\mu(x)| \lesssim\frac{\alpha(Q)}{\ell(Q)} + \|\nabla^2T_m\mu\|_\infty\,\dist(x,L_Q),$$
which yields \rf{derivdic}.

With \rf{derivdic} at hand, the lemma follows easily: given $x,y\in B(z_Q,2 \diam(Q))$, we have
\begin{equation} \label{eqx2}
|T_m\mu(x) - T_m\mu(y)| \leq \sup_u \|\nabla_uT_m\mu\|_{\infty,[x,y]} \,|x-y| + \|\nabla T_m\mu\|_\infty |\Pi_{L_Q^\bot}(x-y)|,
\end{equation}
where the supremum on the right side is taken over all unit vectors parallel to $L_Q$. From~\rf{derivdic} we get
\vspace*{-5pt}
\begin{align*}
\sup_u \|\nabla_uT_m\mu\|_{\infty,[x,y]} & \lesssim \frac{\alpha(Q)}{\ell(Q)} + \frac{\sup_{z\in[x,y]}\dist(z,L_Q)}{\ell(Q)^2} \\
& \leq \frac{\alpha(Q)}{\ell(Q)} + \frac{\dist(x,L_Q)+\dist(y,L_Q)}{\ell(Q)^2}.
\end{align*}
Plugging this estimate into \rf{eqx2} we are finished with the lemma.
\end{proof}

By the preceding result and the definition of $A_{R}$, we have
\begin{equation} \label{a1234}
\begin{split}
|A_{R}|&\lesssim  \sum_{k>J(R)}\sum_{S\in\DD_m:\,S\subset R} \int_{\frac32 S} |T_k\mu|\,\biggl[ \alpha(R)
\frac{\ell(S)}{\ell(R)}  \\ & \quad + \dist(x,L_R)\,
\frac{\ell(S)}{\ell(R)^2} + \dist(x_S,L_R)\,
\frac{\ell(S)}{\ell(R)^2}  +\frac{|\Pi_{L_R^\bot}(x-x_S)|}{\ell(R)}\biggr]\,d\mu(x).
\end{split}
\end{equation}

By Remark \ref{remfa} we have
\begin{equation}\label{dd34}
\dist(x,L_R)\lesssim \dist(x,L_S) + \sum_{P\in \DD:S\subset P\subset R}\alpha(P)\ell(P).
\end{equation}
The same estimate holds if we replace $x$ by $x_S$ (and in this case $\dist(x_S,L_S)=0$).

Now we want to estimate the term $|\Pi_{L_R^\bot}(x-x_S)|$. Let $Q\in\DD_k$ be such that $x\in Q$. We have
$$|\Pi_{L_R^\bot}(x-x_S)| \leq |\Pi_{L_S^\bot}(x-x_S)| + \!\!\!\sum_{P\in \DD:Q\subset P\subset 3R}\!\!\!
|\Pi_{L_P^\bot}(x-x_S) - \Pi_{L_\wh{P}^\bot}(x-x_S)|,$$
where $\wh{P}$ stands for the parent of $P$. Since $x_S\in L_S$, we have $|\Pi_{L_S^\bot}(x-x_S)|=\dist(x,L_S)$.
Moreover, from Lemma \ref{lempr2} it follows easily that
$$|\Pi_{L_P^\bot}(x-x_S) - \Pi_{L_\wh{P}^\bot}(x-x_S)| \lesssim \alpha(\wh{P})|x-x_S|\lesssim \alpha(\wh{P})\ell(S).$$
Therefore, recalling that $m=m(J(R),k)=[(J(R)+k)/2]$,
\begin{equation}\label{sususu}
|\Pi_{L_R^\bot}(x-x_S)|\lesssim \dist(x,L_S)+\sum_{P\in \DD:Q\subset P\subset 3R}\alpha(P)\ell(Q)^{1/2}\ell(R)^{1/2}.
\end{equation}

From \rf{a1234}, 
 \rf{dd34}, and \rf{sususu}, we infer that
\begin{equation} \label{aaa}
\begin{split}
|A_{R}| &\lesssim  \sum_{Q\subset 3R} \int |T_Q\mu(x)| \,\frac{\dist(x,L_S)}{\ell(R)}\,d\mu(x) \\
& \quad + \sum_{Q\subset 3R}
 \frac{\ell(Q)^{1/2}}{\ell(R)^{1/2}}\int |T_Q\mu(x)| \,\Bigl[\sum_{P\in \DD:Q\subset P\subset 3R}\alpha(P)\Bigr]\,d\mu(x) =: A_{R}^1 + A_{R}^2.
\end{split}
\end{equation}
Notice that, although it is not stated explicitly, $S$ depends on $Q$ in $A_R$. In fact, we should 
properly write $S_Q$ instead of $S$.


\subsubsection*{Estimate of $A_{R}^1$ in \rf{aaa}}

For each $Q\subset R$, by Cauchy-Schwartz we have
\begin{align*}
A_{R}^1 &\leq   \sum_{Q\subset 3R} \frac{\ell(S)}{\ell(R)}
\biggl(\int |T_Q\mu|^2 \,d\mu\biggr)^{1/2}
\biggl(\int_Q \frac{\dist(x,L_S)^2}{\ell(S)^2}\,d\mu(x)\biggr)^{1/2}\\
&\lesssim  \sum_{Q\subset 3R} \frac{\ell(Q)^{1/2}}{\ell(R)^{1/2}}
 \int |T_Q\mu|^2 \,d\mu + \sum_{S\subset 3R}\frac{\ell(S)}{\ell(R)}
\int_S \frac{\dist(x,L_S)^2}{\ell(S)^2}\,d\mu(x).
\end{align*}
Using Lemmas \ref{lempr3} and \ref{lemdif1}, we get
\begin{align*}
\sum_{R\in \DD} A_{R}^1 &\lesssim  \sum_{R\in \DD}  \sum_{Q\subset 3R} \frac{\ell(Q)^{1/2}}{\ell(R)^{1/2}} {\mathcal A}^2(Q)\,\mu(Q) +
\sum_{R\in \DD}  \sum_{S\subset 3R} \frac{\ell(S)}{\ell(R)} {\mathcal A}^2(S)\,\mu(S) \\
&\lesssim  \sum_{Q\in \DD} {\mathcal A}^2(Q)\,\mu(Q) \lesssim \sum_{Q\in \DD} \alpha(Q)^2\,\mu(Q).
\end{align*}


\subsubsection*{Estimate of $A_{R}^2$ in \rf{aaa}}

By Lemma \ref{lempr3} and Cauchy-Schwartz, we have
\begin{equation}\label{a22}
\begin{split}
A_{R}^2 &\lesssim 
\sum_{Q\subset 3R}
 \frac{\ell(Q)^{1/2}}{\ell(R)^{1/2}} \Bigl[\sum_{P\in \DD:Q\subset P\subset 3R}\alpha(P)\Bigr]^2\mu(Q)\\
&\lesssim  \sum_{Q\subset 3R}
 \frac{\ell(Q)^{1/2}}{\ell(R)^{1/2}}\,\log\biggl(2+\frac{\ell(R)}{\ell(Q)}\biggr)
\sum_{P\in \DD:Q\subset P\subset 3R}\alpha(P)^2 \mu(Q). 
\end{split}
\end{equation}
Thus,
\begin{align*}
\sum_{R\in \DD} A_{R}^2 &\lesssim  \sum_{R\in \DD} 
\sum_{Q\subset 3R}
 \frac{\ell(Q)^{1/3}}{\ell(R)^{1/3}}
\sum_{P\in \DD:Q\subset P\subset 3R}\alpha(P)^2 \mu(Q)\\
&=  \sum_{P\in \DD} \alpha(P)^2 \sum_{Q:Q\subset P} \mu(Q)\sum_{R:3R\supset P}  \frac{\ell(Q)^{1/3}}{\ell(R)^{1/3}}\\
&\lesssim  \sum_{P\in \DD} \alpha(P)^2 \sum_{Q:Q\subset P} \mu(Q)\,\frac{\ell(Q)^{1/3}}{\ell(P)^{1/3}} \lesssim 
\sum_{P\in \DD} \alpha(P)^2 \mu(P).
\end{align*}



\subsection{Estimates for \boldmath$B_{j,k}$ in \rf{eqAB}}

Recall that for $j,k$, with $k>j$, we have
$$B_{j,k} = \sum_{S\in \DD_m}  T_j\mu(x_S) \langle \varphi_S,T_k\mu\rangle,$$
where $m=[(j+k)/2]$.
We say that two cubes $S,T\in\DD_m$ (of the same generation $m$) are neighbors if $\dist(S,T)\leq 2^{-m}$ and $S\neq T$, 
and then we write $S\in N(T)$ and $T\in N(S)$.
The number of neighbors of a given $S\in\DD$ is bounded above independently of $S$, i.e.\
there is some constant $C$ such that $\# \{T\in\DD:T\in N(S)\} \leq C$ for all $S\in\DD$.
From the fact that $\supp(K_k)\subset B(0,2^{-k-2})$ and the
antisymmetry of $T_k$ we infer that, for $S,T\in\DD$ of the same generation, $\langle \varphi_S,T_k(\varphi_T  \mu)\rangle =0$ unless $S$ and $T$ are neighbors.
So we have
\begin{equation} \label{mit1}
B_{j,k} = \sum_{S\in \DD_m} \sum_{T\in N(S)} T_j\mu(x_S)\langle \varphi_S,T_k(\varphi_T\mu)\rangle.
\end{equation}
Using the antisymmetry of $T_k$, reordering the sums, and interchanging the notation of $S$ and~$T$ we get
\begin{equation}\label{mit2}
\begin{split}
B_{j,k} & = - \sum_{S\in \DD_m} \sum_{T\in N(S)} T_j\mu(x_S)\langle \varphi_T,T_k(\varphi_S\mu)\rangle \\
& = - \sum_{T\in \DD_m} \sum_{S\in N(T)} T_j\mu(x_S)\langle \varphi_T,T_k(\varphi_S\mu)\rangle \\
& = - \sum_{S\in \DD_m} \sum_{T\in N(S)} T_j\mu(x_T)\langle \varphi_S,T_k(\varphi_T\mu)\rangle.
\end{split}
\end{equation}
If we take the mean value of \rf{mit1} and \rf{mit2} we obtain
\begin{equation} \label{eqabbb}
\begin{split}
B_{j,k} &=  \frac12 \sum_{S\in \DD_m} \sum_{T\in N(S)} \bigl(T_j\mu(x_S)-T_j\mu(x_T)\bigr)\langle \varphi_S,T_k(\varphi_T\mu)\rangle\\
& = \frac12 \sum_{S\in \DD_m} \sum_{T\in N(S)} \bigl(T_j\mu(x_S)-T_j\mu(x_T)\bigr) \Bigl[
\langle \varphi_S,T_k(\varphi_T\mu)\rangle - \langle \varphi_S,T_k(\varphi_T\LL_S)\rangle_{\LL_S}\Bigr]\\
& \quad +
\frac12 \sum_{S\in \DD_m} \sum_{T\in N(S)} \bigl(T_j\mu(x_S)-T_j\mu(x_T)\bigr)
\langle \varphi_S,T_k(\varphi_T\LL_S)\rangle_{\LL_S} \\
&= \frac12 \bigl(B_{j,k}^1 + B_{j,k}^2\bigr),
\end{split}
\end{equation} 
where we used the notation $\langle f,g\rangle_{\LL_S} = \int fg\,d\LL_S$.


\subsubsection{Estimates for $B_{j,k}^2$ in \rf{eqAB}}

Since $\langle \varphi_S,T_k(\varphi_S\LL_S)\rangle_{\LL_S}=0$ and 
$T_k(\LL_S)$ vanishes identically on $L_S$, we have
\begin{equation}\label{mit4}
\begin{split}
\sum_{S\in \DD_m} \sum_{T\in N(S)} T_j\mu(x_S) &
\langle \varphi_S,T_k(\varphi_T\LL_S)\rangle_{\LL_S} \\ & = \sum_{S\in \DD_m} T_j\mu(x_S)
\langle \varphi_S,T_k(\LL_S)\rangle_{\LL_S} = 0.
\end{split}
\end{equation}
Interchanging the roles of $S$ and $T$, and using the antisymmetry of $T_k$ we also get
\begin{equation} \label{mit5}
\begin{split}
0&= \sum_{T\in \DD_m} \sum_{S\in N(T)} T_j\mu(x_T) \langle \varphi_T,T_k(\varphi_S\LL_T)\rangle_{\LL_T} \\
&= - \sum_{T\in\DD_m} \sum_{S\in N(T)}  T_j\mu(x_T) \langle \varphi_S,T_k(\varphi_T\LL_T)\rangle_{\LL_T}.
\end{split}
\end{equation}
Thus, if we plug \rf{mit4} and \rf{mit5} into the definition of $B_{j,k}^2$, we get
\begin{equation} \label{eqff88}
\begin{split}
B_{j,k}^2 & = 
-\sum_{S\in \DD_m} \sum_{T\in N(S)} T_j\mu(x_T)
\langle \varphi_S,T_k(\varphi_T\LL_S)\rangle_{\LL_S} \\
& = 
\sum_{S\in \DD_m} \sum_{T\in N(S)} 	T_j\mu(x_T) \Bigl[
\langle \varphi_S,T_k(\varphi_T\LL_T)\rangle_{\LL_T} -
\langle \varphi_S,T_k(\varphi_T\LL_S)\rangle_{\LL_S}\Bigr].
\end{split}
\end{equation}

\begin{claim}\label{cla1}
	For all $S,T\in\DD_m$ which are neighbors, we have
$$\Bigl|\langle \varphi_S,T_k(\varphi_T\LL_T)\rangle_{\LL_T} -
\langle \varphi_S,T_k(\varphi_T\LL_S)\rangle_{\LL_S}\Bigr| \lesssim 2^{-|j-k|/2}\alpha(S) \ell(S)^n.$$
\end{claim}

\begin{proof}
Take $S,T\in\DD_m$ which are neighbors. 
To prove the claim we may assume that $\alpha(S)$ is small enough, so that the angle between $L_T$
and $L_S$ is $\leq\pi/4$. This is due to the fact that
$$\|T_k(\varphi_T\LL_T)\|_{\infty,L_T} +
\|T_k(\varphi_T\LL_S)\|_{\infty,L_S} \lesssim 2^{-|j-k|/2}.$$
See \rf{eqff2} and \rf{eqff8} below for some details.
 
Let $p:L_S\rightarrow L_T$ be the orthogonal projection from $L_S$ into $L_T$. 
Let $p^{-1}\LL_T$ be the image measure of $\LL_T$ by $p^{-1}$. So $p^{-1}\LL_T$ is a multiple of the $n$-dimensional Lebesgue measure
on $L_S$. It is easy to check that $p^{-1}\LL_T = a \LL_S$, where $a$~is some constant such that $|a-1| \lesssim \alpha(S)$.
Then we have
\begin{equation}\label{eqff0}
\begin{split}
\langle \varphi_S,T_k(\varphi_T\LL_T)\rangle_{\LL_T} & = \int \varphi_S T_k(\varphi_T\LL_T)\,d\LL_T \\
& = a \int (\varphi_S\circ p) \bigl(T_k(\varphi_T\LL_T)\circ p\bigr)\,d\LL_S.
\end{split}
\end{equation}

We split the inequality in the claim as follows:
\begin{equation} \label{eqff01}
\begin{split}
\Bigl| \langle \varphi_S,\, &T_k(\varphi_T\LL_T)\rangle_{\LL_T} -
\langle \varphi_S,T_k(\varphi_T\LL_S)\rangle_{\LL_S}\Bigr| \\
\leq &\; 
\Bigl|\langle \varphi_S,T_k(\varphi_T\LL_T)\rangle_{\LL_T} -
a \int (\varphi_S\circ p) 
T_k\bigl((\varphi_T\circ p)\LL_S\bigr)\,d\LL_S\Bigr|  \\
& \mbox{} + \Bigl|a \int (\varphi_S\circ p) T_k\bigl((\varphi_T\circ p) \,\LL_S\bigr)\,d\LL_S
- a \int (\varphi_S\circ p) T_k(\varphi_T\LL_S)\,d\LL_S\Bigr|  \\
& \mbox{} + \Bigl|a \int (\varphi_S\circ p) T_k(\varphi_T\LL_S)\,d\LL_S 
- a \int \varphi_S T_k(\varphi_T\LL_S)\,d\LL_S \Bigr|  \\
& \mbox{} + |a-1|\, \bigl|\langle \varphi_S,T_k(\varphi_T\LL_S)\rangle_{\LL_S}\bigr| =: D_1+ D_2 + D_3 + D_4.
\end{split}
\end{equation}
We will show that each of the four terms on the right hand side of the above inequality is~$\lesssim 
2^{-|j-k|/2}\alpha(S) \mu(S)$.

For the first term, by \rf{eqff0}, it is enough to show that if $x\in L_S$, then
\begin{equation}\label{ff1}
\bigl|\bigl(T_k(\varphi_T\LL_T)\circ p\bigr)(x) - T_k\bigl((\varphi_T\circ p)\LL_S\bigr)(x)\bigr| \lesssim 2^{-|j-k|/2}\alpha(S).
\end{equation}
To this end we set
\begin{align*}
\bigl(T_k(\varphi_T\LL_T)\circ p\bigr)(x) & = \int K_k(p(x)-y) \varphi_T(y) \,d\LL_T(y)\\
& = a \int K_k(p(x)-p(y)) \varphi_T(p(y)) \,d\LL_S(y).
\end{align*}
Since $p$ is an affine map, it can be written as $p= b+\tilde p$, where $b$ is some constant and $\tilde p$ is linear. So we have
$K_k(p(x)-p(y)) = \bigl(K_k\circ \tilde p\bigr)(x-y)$. Using the oddness of $K_k\circ \tilde p$ (recall also that $x\in L_S$), we get
\begin{align*}
\bigl(T_k(\varphi_T\LL_T)\circ p\bigr)(x) & =  a \int \bigl(K_k\circ \tilde p\bigr)(x-y) \varphi_T(p(y)) \,d\LL_S(y)\\
& = a \int \bigl(K_k\circ \tilde p\bigr)(x-y) \bigl(\varphi_T(p(y))-\varphi_T(p(x))\bigr) \,d\LL_S(y).
\end{align*}
Again by the oddness of $K_k$,
\begin{equation}\label{eqff1}
T_k\bigl((\varphi_T\circ p)\LL_S\bigr)(x) = \int K_k(x-y) \bigl(\varphi_T(p(y)) - \varphi_T(p(x))\bigr)\,d\LL_S(y).
\end{equation}
Therefore, 
\begin{equation} \label{eqff2}
\begin{split}
\bigl|\bigl(T_k(\varphi_T\LL_T)\circ p\bigr)(x) &  - T_k\bigl((\varphi_T\circ p)\LL_S\bigr)(x)\bigr|  \\
\leq &\; |a-1| \bigl|T_k\bigl((\varphi_T\circ p)\LL_S\bigr)(x)\bigr| \\
& \mbox{}+
|a| \int \bigl|K_k\circ \tilde p - K_k\bigr|(x-y) \,\bigl|\varphi_T(p(y))-\varphi_T(p(x))\bigr| \,d\LL_S(y).
\end{split}
\end{equation}
To estimate the term $\bigl|T_k\bigl((\varphi_T\circ p\bigr)\LL_S)(x)\bigr|$ we use the identity \rf{eqff1}. Since
$\supp(K_k)\subset B(0,2^{-k-2})$, for $x\in L_S$ we derive 
\begin{equation}\label{eqff3}
\begin{split}
\bigl|T_k\bigl((\varphi_T\circ p)\LL_S\bigr)(x)\bigr|& \leq \|\nabla (\varphi_T\circ p)\|_\infty 2^{-k} 
\int |K_k(x-y)| \,d\LL_S(y) \\ &\lesssim \frac{2^{-k}}{\ell(S)}\approx 2^{-|j-k|/2}.
\end{split}
\end{equation}
To estimate the last integral in \rf{eqff2}, we take into account that
$$\bigl|K_k\circ \tilde p - K_k\bigr|(x-y) \lesssim \|\nabla K_k\|_\infty |\tilde p(x-y)-(x-y)| \lesssim \frac{\|\tilde p -I\| 
|x-y|}{2^{-k(n+1)}}.$$
It is easy to check that $\|\tilde p -I\|\lesssim \alpha(S)$. Moreover, we can assume 
$|y-x|\lesssim 2^{-k}$. So we get $\bigl|K_k\circ \tilde p - K_k\bigr|(x-y)\lesssim \alpha(S) 2^{kn}$. Thus
\begin{align*}
\int \bigl|K_k\circ \tilde p \;-\; & K_k\bigr|(x-y) \bigl|\varphi_T(p(y))- \varphi_T(p(x))\bigr| \,d\LL_S(y) \\
 &\lesssim \alpha(S) \|\nabla (\varphi_T\circ p)\|_\infty 2^{-k} 
\lesssim \alpha(S) \,\frac{2^{-k}}{\ell(S)}\approx 2^{-|j-k|/2}\alpha(S).
\end{align*}
So \rf{ff1} follows from \rf{eqff2}, \rf{eqff3}, the last estimate, and the
fact that $|a-1|\lesssim\alpha(S)$.

Let us turn our attention to the term $D_2$ in \rf{eqff01}. Let us denote
$f(y) = \varphi_T(p(y)) - \varphi_T(y)$. By the oddness of $K_k$, for $x\in L_S$, we have
$$T_k\bigl((\varphi_T\circ p) \,\LL_S\bigr) (x) - T_k(\varphi_T\LL_S)(x)= \int
K_k(x-y) \bigl(f(y) - f(x)\bigr) \,d\LL_S(y).$$
Thus,
$$\bigl|T_k\bigl((\varphi_T\circ p) \,\LL_S\bigr) (x) - T_k(\varphi_T\LL_S)(x)\bigr| \lesssim \|\nabla f\|_\infty 2^{-k},$$
and so
\begin{equation}\label{eqff15}
|D_2| \lesssim \|\nabla f\|_\infty 2^{-k}\ell(S)^n.
\end{equation}
We will show that $\|\nabla f\|_\infty\lesssim \alpha(S)/\ell(S)$, and we will be done with $D_2$. Indeed, we have
\begin{align*}
|\nabla f(x)| &=  \bigl|\nabla\varphi_T(p(x))\,\cdot \tilde p - \nabla\varphi_T(x)\bigr| \\ &\leq 
\bigl|\nabla\varphi_T(p(x))\bigr| \,\|I - \tilde p\| + \bigl|\nabla\varphi_T(p(x)) -
\nabla\varphi_T(x)\bigr| \\
&\lesssim \frac{1}{\ell(S)}\,\alpha(S) + \|\nabla^2\varphi_T\|_\infty \,|x-p(x)| \lesssim \frac{\alpha(S)}{\ell(S)},
\end{align*}
as promised.

To deal with the term $D_3$ in \rf{eqff01}, we notice that for $x\in L_S$ we have
\begin{equation}\label{eqff8}
|T_k(\varphi_T\LL_S)(x)| \lesssim 2^{-|j-k|/2}.
\end{equation}
The proof is analogous to the one for \rf{eqff3}. We also have
\begin{equation}\label{eqff85}
|(\varphi_S\circ p)(x) - \varphi_S(x)|\leq \|\nabla \varphi_S\|_\infty |p(x) - x| \lesssim \frac{1}{\ell(S)}\,\alpha(S)\ell(S) = \alpha(S).
\end{equation}
From \rf{eqff8} and the preceding estimate it follows that $D_3\lesssim 2^{-|j-k|/2}\alpha(S)\mu(S)$.

Finally, the estimate for the term $D_4$ in  \rf{eqff01} follows from \rf{eqff3} and the fact that $|a-1|\lesssim\alpha(S)$.
\end{proof}

We are ready to estimate $\sum_{k>j} B_{j,k}^2$ now. By \rf{eqff88} and Claim \ref{cla1}, we have
$$\sum_{j,k:k>j} |B_{j,k}^2| \lesssim \sum_{j,k:k>j} \,
\sum_{S\in \DD_{m(j,k)}} \sum_{T\in N(S)} |T_j\mu(x_T)|\, 2^{-|j-k|/2}\alpha(S) \ell(S)^n.$$
Since $x_T\in T\cap L_T$, we have $|T_j\mu(x_T)|\lesssim\alpha(T)\lesssim \alpha(\wh S)$ for $T\in N(S)$,
where $\wh S$ denotes the parent (or a suitable ancestor) of $S$. Thus,
$$\sum_{j,k:k>j} |B_{j,k}^2| \lesssim \sum_{j,k:k>j} \,
\sum_{S\in \DD_{m(j,k)}}  2^{-|j-k|/2}\alpha(\wh S)^2 \mu(S).$$
Recalling that $m(j,k) = [(j+k)/2]$, we get
\begin{align*}
\sum_{j,k:k>j} |B_{j,k}^2|  & \lesssim \sum_{j\in\Z} \sum_{S\in \DD:\ell(S)\leq 2^{-j}}  \frac{\ell(S)}{2^{-j}}\,\alpha(\wh S)^2 \mu(S)\\
& = \sum_{S\in \DD} \alpha(\wh S)^2 \mu(S) \sum_{j:2^{-j} \geq \ell(S)} \frac{\ell(S)}{2^{-j}} \lesssim \sum_{S\in \DD} \alpha(S)^2 \mu(S).
\end{align*}


\subsubsection{Estimates for $B_{j,k}^1$ in \rf{eqabbb}}

We have
$$B_{j,k}^1 = \sum_{S\in \DD_m} \sum_{T\in N(S)} \bigl(T_j\mu(x_S)-T_j\mu(x_T)\bigr) \Bigl[
\langle \varphi_S,T_k(\varphi_T\mu)\rangle - \langle \varphi_S,T_k(\varphi_T\LL_S)\rangle_{\LL_S}\Bigr].$$
Let us estimate the difference $\langle \varphi_S,T_k(\varphi_T\mu)\rangle - \langle \varphi_S,T_k(\varphi_T\LL_S)\rangle_{\LL_S}$. Let $\{\vphi_Q\}_{Q\in\DD_k}$ be a partition of unity 
with $\vphi_Q\in\CC^\infty$,  $\|\nabla \vphi_Q\|_\infty\lesssim1/\ell(Q)$, and $\supp(\vphi_Q)\subset 2Q$ for each $Q\in\DD_k$, and set $\psi_Q= \vphi_Q\vphi_S$. We have
\begin{align}\label{eqs1s2}
 \langle \varphi_S,T_k(\varphi_T\mu)\rangle - \langle \varphi_S,T_k(\varphi_T\LL_S)\rangle_{\LL_S} & = 
\sum_{Q\in\DD_k} \Bigl(
\langle \psi_Q,T_k(\varphi_T\mu)\rangle - \langle \psi_Q,T_k(\varphi_T\LL_S)\rangle_{\LL_S}\Bigr)\\
& = 
\sum_{Q\in\DD_k} \!\!\Bigl(
\langle \psi_Q,T_k(\varphi_T\mu)\rangle - \langle \psi_Q,T_k(\varphi_T\LL_Q)\rangle_{\LL_Q}\Bigr)\nonumber\\
& 
+\!
\sum_{Q\in\DD_k} \!\!\Bigl(
\langle \psi_Q,T_k(\varphi_T\LL_Q)\rangle_{\LL_Q}\! - \langle \psi_Q,T_k(\varphi_T\LL_S)\rangle_{\LL_S}\!\Bigr)\!
=: S_1 + S_2. \nonumber
\end{align}

First we consider the sum $S_1$.
By the definition of $\alpha(Q)$, for $x\in 2Q$ we have
\begin{equation} \label{eqff55}
\begin{split}
\bigl|T_k (\varphi_T\mu)(x)&  - T_k(\varphi_T\LL_Q)(x)\bigr| \\ = & \;\biggl|\int K_k(x-y)\varphi_T(y)\,d\mu(y) - \!
\int K_k(x-y)\varphi_T(y)\,d\LL_Q(y)\biggr|
\lesssim \alpha(Q),
\end{split}
\end{equation}
since $\supp(K_k(x-\cdot)\varphi_T)\subset B_Q$ and $\|\nabla \bigl(K_k(x-\cdot)\varphi_T\bigr)\|_\infty \lesssim 1/\ell(Q)^{n+1}$.
Now we write
\begin{align*}
\bigl|\langle &\psi_Q,T_k(\varphi_T\mu)\rangle - \langle \psi_Q,T_k(\varphi_T\LL_Q)\rangle_{\LL_Q}\bigr| \\
\leq & \; \bigl|\langle  \psi_Q,T_k(\varphi_T\mu)\rangle \!- \!\langle \psi_Q,T_k(\varphi_T\LL_Q)\rangle\bigr| +
\bigl|\langle \psi_Q,T_k(\varphi_T\LL_Q)\rangle \!- \!\langle \psi_Q,T_k(\varphi_T\LL_Q)\rangle_{\LL_Q}\bigr|. 
\end{align*}
By \rf{eqff55} the first term on the right side is $\lesssim \alpha(Q)\mu(Q)$. By Fubini, the second term on the right
side equals
$$\bigl|\langle \varphi_T,T_k(\psi_Q\mu)\rangle_{\LL_Q} \!- \!\langle \varphi_T,T_k(\psi_Q\LL_Q)\rangle_{\LL_Q}\bigr|,$$
which by \rf{eqff55} is also $\lesssim \alpha(Q) \mu(Q)$. Thus,
\begin{equation}\label{eqs1}
|S_1| \lesssim \sum_{Q\in \DD_k:Q\subset3S} \alpha(Q) \mu(Q).
\end{equation}

Now we consider the sum $S_2$.

\begin{claim}\label{cla2}
	For all $S,T\in\DD_m$ which are neighbors and $Q\subset 3S$, we have
$$\Bigl|
\langle \psi_Q,T_k(\varphi_T\LL_Q)\rangle_{\LL_Q} -
\langle \psi_Q,T_k(\varphi_T\LL_S)\rangle_{\LL_S}\Bigr| \lesssim 
\sum_{P:Q\subset P\subset 3S} \alpha(P)\,\frac{\ell(P)}{\ell(S)}\,\mu(Q).
$$
\end{claim}

\begin{proof}
The estimates are very similar to the ones in Claim \ref{cla1}, and so we only give some hints:
we assume that $\sum_{P:Q\subset P\subset 3S} \alpha(P)\,\frac{\ell(P)}{\ell(S)}$ is small enough and
we consider the orthogonal projection $p$ from $L_Q$ into $L_T$. Then we
split the term $\langle \psi_Q,T_k(\varphi_T\LL_Q)\rangle_{\LL_Q} -
\langle \psi_Q,T_k(\varphi_T\LL_S)\rangle_{\LL_S}$ like in \rf{eqff01}, so that we obtain terms analogous 
to $D_1,\ldots,D_4$. The new estimates for $D_1,D_2,D_4$ are very similar to the ones in the proof of Claim \ref{cla1}. The main difference is that now we have $\|p-I\|\lesssim \sum_{P:Q\subset P\subset S} \alpha(P).$

For the term $D_3$, inequality \rf{eqff85} should be replaced by the following:
$$|(\psi_Q\circ p)(x) - \psi_Q(x)|\leq \|\nabla \psi_Q\|_\infty |p(x) - x| \lesssim \frac{1}{\ell(Q)}\,
\sum_{P:Q\subset P\subset 3S} \alpha(P)\,\ell(P),$$
and then, by \rf{eqff8},
$$|D_3| \lesssim \frac{1}{\ell(S)}\,
\sum_{P:Q\subset P\subset 3S} \alpha(P)\,\ell(P)\mu(Q).$$
\end{proof}

So we have
\begin{equation}\label{eqs2}
S_2 \lesssim \sum_{Q\in \DD_k:Q\subset3S}    \sum_{P:Q\subset P\subset 3S} \alpha(P)\,\frac{\ell(P)}{\ell(S)}\,\mu(Q)
\lesssim  \sum_{P\subset 3S} \alpha(P)\,\frac{\ell(P)}{\ell(S)}\,\mu(P) \lesssim \AZ(S)\mu(S).
\end{equation}

On the other hand, if we denote by $R$ the cube in the generation $j$ which contains $S$, by Lemma \ref{lempr4} we have
\begin{align*}
\bigl|T_j\mu(x_S)-T_j\mu(x_T)\bigr| &\lesssim 
\frac{\alpha(R)\ell(R)+
\dist(x_S,L_R) + \dist(x_T,L_R)}{\ell(R)^2}\,\ell(S) \\ &\quad+
\frac{|\Pi_{L_R^\bot}(x_S-x_T)|}{\ell(R)}.
\end{align*}
Therefore, by \rf{eqs1} and \rf{eqs2},
\begin{align*}
|B_{j,k}^1| &\lesssim  \sum_{R\in \DD_j}\,\sum_{S\in \DD_m:S\subset R} 
\Bigl[\sum_{Q\in \DD_k:Q\subset3S} \alpha(Q) \mu(Q)+    \AZ(S)\mu(S)\Bigr]
\\ &\quad \times \sum_{T\in N(S)}  \biggl[\alpha(R)\frac{\ell(S)}{\ell(R)}
+\dist(x_S,L_R)\, \frac{\ell(S)}{\ell(R)^2}+ \dist(x_T,L_R)\, \frac{\ell(S)}{\ell(R)^2} +
\frac{|\Pi_{L_R^\bot}(x_S-x_T)|}{\ell(R)}\biggr].
\end{align*}
We have
$$\dist(x_S,L_R) + \dist(x_T,L_R) \lesssim \sum_{P:S\subset P\subset R} \alpha(P)\ell(P),$$
and also, arguing as in \rf{sususu},
$$|\Pi_{L_R^\bot}(x_S-x_T)|\lesssim \sum_{P:S\subset P\subset R}\alpha(P)\ell(S).$$
Therefore,
$$|B_{j,k}^1| \lesssim  \sum_{R\in \DD_j}\,\sum_{S\in \DD_m:S\subset R} 
\Bigl[\sum_{Q\in \DD_k:Q\subset3S} \alpha(Q) \mu(Q)+    \AZ(S)\mu(S)\Bigr] \sum_{P:S\subset P\subset R}\alpha(P)\frac{\ell(S)}{\ell(R)},$$
and so
\begin{align*}
\sum_{j,k} |B_{j,k}^1| & \lesssim \sum_{R\in \DD}\,
\sum_{Q\subset R} \alpha(Q) \mu(Q) \!\!\!\sum_{P:Q\subset P\subset R}\alpha(P)\frac{\ell(Q)^{1/2}}{\ell(R)^{1/2}} +
\sum_{R\in \DD}\,\sum_{S\subset R} 
 \AZ(S)\mu(S)\!\!\!\sum_{P:S\subset P\subset R}\alpha(P)\frac{\ell(S)}{\ell(R)} \\ & \lesssim
\sum_{R\in \DD}\,\sum_{S\subset R} 
 \AZ(S)\mu(S)\!\!\!\sum_{P:S\subset P\subset R}\alpha(P)\frac{\ell(S)^{1/2}}{\ell(R)^{1/2}}\\
& \approx
\sum_{S\in\DD} 
 \AZ(S)\mu(S)\sum_{P:S\subset P}\alpha(P)\frac{\ell(S)^{1/2}}{\ell(P)^{1/2}}
\end{align*}
By Cauchy-Schwartz we obtain,
$$\sum_{j,k} |B_{j,k}^1|\lesssim \Bigl(
\sum_{S\in\DD} 
 \AZ(S)^2\mu(S)\Bigr)^{1/2}
\Bigl(\sum_{P\in\DD}\alpha(P)^2\mu(P)\Bigr)^{1/2} \lesssim \sum_{S\in\DD}\alpha(S)^2\mu(S).$$


\section{Riesz transforms and quasiorthogonality} \label{secqo}

In this section we will prove Theorem \ref{teoqo}. First we introduce the functions $\vphi_m$ that are used to
define the kernels of the doubly truncated Riesz transforms.

\begin{definition}\label{defpsifi}
Let $\psi:[0,+\infty)\to[0,+\infty)$ be a non increasing $\CC^2$ function such that $\chi_{[0,1/4]}\leq\psi\leq
\chi_{[0,4]}$. Suppose moreover that $|\psi'|$ is bounded below away from zero in $[1/3,3]$. That is to 
say, 
\begin{equation}\label{cond17}
\chi_{[1/3,3]} \leq C_5|\psi'|.
\end{equation}
For $m\in\Z$ and $x\in\R^d$ denote $\rho_m(x) = 1 - \psi\bigl(2^{2m}|x|^2\bigr),$
and
$$\vphi_m(x) = \psi\bigl(2^{2m}|x|^2\bigr) - \psi\bigl(2^{2m+2}|x|^2\bigr).$$
We set
\begin{equation}\label{defrm}
R_m\mu(x) = \int \vphi_m(x-y)\,\frac{x-y}{|x-y|^{n+1}}\,d\mu(y).
\end{equation}
\end{definition}

Notice that $\supp(\rho_m)\subset \R^d\setminus B(0,2^{-m-1})$ and
$\supp(\vphi_m)\subset A(0,2^{-m-2},2^{-m+1})$.
Moreover, 
$$\sum_{m\in\Z} \vphi_m(x)=1 \mbox{ \,for all $x\neq 0$},$$
and so, formally,
$$\sum_{m\in\Z} R_m\mu(x) = R\mu(x),$$
where $R\mu$ stands for the $n$-dimensional Riesz transform.


\subsection{Preliminary lemmas}

Given a function $\vphi:[0,+\infty) \to[0,+\infty)$ and $\ve>0$, we denote
$$R_{\vphi,\ve}\mu(x) = \int \vphi\biggl(\frac{|x-y|^2}{\ve^2}\biggr)  \frac{x-y}{|x-y|^{n+1}}\,d\mu(y).$$
For the applications below one should think that $\vphi$ is of the form
$$\vphi(t) = 1 - \psi(t)  \quad\mbox{ or }\quad \vphi(t) = \psi(t) - \psi(t/4),$$
where $\psi$ is the function introduced in Definition \ref{defpsifi}. If $\ve= 2^{-m}$, in the first case we have
$\vphi(|x|^2/\ve^2) = \rho_m(x),$ and in the second one,
$\vphi(|x|^2/\ve^2) = \vphi_m(x).$

\begin{lemma}\label{lemtaylor}
Let $\vphi:[0,+\infty)\to[0,+\infty)$ be a $\CC^2$ function with $\supp(\vphi)\subset [1/4,\,+\infty)$ and 
$\supp(\vphi')\subset [1/4,\,4]$. Let $\ve>0$ and let $x\in\R^{d}$ such that $|x|\leq \ve/4$. We have
\begin{equation}\label{eqte}
R_{\vphi,\ve}\mu(x) - R_{\vphi,\ve}\mu(0) = T(x) + E(x),
\end{equation}
with
\begin{equation}\label{eqte2}
T(x) = \int\frac1{|y|^{n+1}}\,\biggl[
\vphi\biggl(\frac{|y|^2}{\ve^2}\biggr) \biggl(x-\frac{(n+1)(x\cdot y)y}{|y|^2}\biggr) + 
\vphi'\biggl(\frac{|y|^2}{\ve^2}\biggr) \frac{2(x\cdot y)y}{\ve^2}
\biggr] d\mu(y),
\end{equation}
and
$$|E(x)|\leq C_6\,\frac{|x|^2}{\ve^2}.$$
Given a unitary vector $v$, if $\supp(\vphi)\subset [1/4,4]$,
then we have
\begin{equation}\label{eqev}
|E(x)\cdot v| \leq C_6\biggl(\frac{|x|\,|x\cdot v|}{\ve^2} + \frac{|x|^2}{\ve^{n+3}} \int_{B(0,2\ve)}|y\cdot v|\,d\mu(y)\biggr).
\end{equation}
The constant $C_6$ only depends on $\|\vphi^{(k)}\|_\infty$, $k=0,1,2$.
\end{lemma}

\begin{proof}
The lemma follows by a direct application of Taylor's formula. Indeed, let $g(s) = \varphi(s)/s^{(n+1)/2}$. By Taylor's formula,
$$\frac{\varphi(s)}{s^{(n+1)/2}} = \frac{\varphi(s_0)}{s_0^{(n+1)/2}} + \frac{s_0\varphi'(s_0) -\frac{n+1}2
\varphi(s_0)}{s_0^{(n+3)/2}}\,(s-s_0) + g''(\xi)\,\frac{(s-s_0)^2}2,$$
for some $\xi\in [s,s_0]$. If we set
$s= |x-y|^2/\ve^2$ and $s_0=|y|^2/\ve^2$, and we multiply by $(x-y)$, we get
\begin{align}\label{eqll1}
\vphi\biggl(\frac{|x-y|^2}{\ve^2}\biggr)  \frac{x-y}{|x-y|^{n+1}} & = 
\vphi\biggl(\frac{|y|^2}{\ve^2}\biggr)  \frac{-y}{|y|^{n+1}} + \vphi\biggl(\frac{|y|^2}{\ve^2}\biggr)  \frac{x}{|y|^{n+1}}  \\
& \quad + \frac{\frac{|y|^2}{\ve^2}\varphi'\left(\frac{|y|^2}{\ve^2}\right) - \frac{n+1}2 \varphi\!\left(\frac{|y|^2}{\ve^2}\right) }{|y|^{n+3}}\,\bigl(|x|^2 - 2x\cdot y\bigr)(x-y) \nonumber\\
&\quad + g''(\xi_{x,y})\,\frac{\bigl(|x|^2 - 2x\cdot y\bigr)^2}{2\ve^{n+5}}\,(x-y), \nonumber
\end{align}
where $\xi_{x,y}\in \bigl[|y|^2/\ve^2,\,|x-y|^2/\ve^2\bigr]$.
If we integrate with respect to $\mu$ and $y$, we obtain \rf{eqte}, with
$E(x) = \int E(x,y)d\mu(y)$, where
\begin{align*}
E(x,y) &= \frac1{|y|^{n+3}}\,\biggl[\frac{|y|^2}{\ve^2}
\varphi'\left(\frac{|y|^2}{\ve^2}\right)
- \frac{n+1}2 \varphi\!\left(\frac{|y|^2}{\ve^2}\right)
\biggr]
\bigl[|x|^2x - |x|^2y - 2(x\cdot y)x \bigr] \\
& \quad + g''(\xi_{x,y})\, \frac{\bigl(|x|^2 - 2x\cdot y\bigr)^2}{2\ve^{n+5}}\,(x-y) =: E_1(x,y) + E_2(x,y).
\end{align*}

Now we have to estimate the term $E(x)$. From the assumptions on $\vphi$, we have
$$\biggl|\varphi\!\left(\frac{|y|^2}{\ve^2}\right)\biggl|+ \frac{|y|^2}{\ve^2} \biggl|
\varphi'\left(\frac{|y|^2}{\ve^2}\right)\biggl| \leq C,$$
and then it easily follows that
$$|E_1(x,y)| \lesssim \frac{|x|^2}{|y|^{n+2}}.$$
Now we deal with $E_2(x,y)$. We have
$$g''(s) = \frac{s^2\varphi''(s) - (n+1)s\,\varphi'(s) + \frac{(n+1)(n+3)}4 \varphi(s)}{s^{(n+5)/2}}.$$
By the properties of $\varphi$, we have $|g''(s)|\lesssim 1/{|s|^{(n+5)/2}}.$
To estimate $g''(\xi_{x,y})$ we may assume that $\xi_{x,y}>1/4$ since otherwise $g''(\xi_{x,y})=0$. 
Recall that $\xi_{x,y}\in \bigl[|y|^2/\ve^2,\,|x-y|^2/\ve^2\bigr]$, and so it easily follows that the condition $\xi_{x,y}>1/4$ implies that 
$|y|\geq \ve/4$, and then $\xi_{x,y}\approx|y|^2/\ve^2$.
Thus,
$|g''(\xi_{x,y})|\lesssim \ve^{n+5}/|y|^{n+5}$ in any case, and then
$$|E_2(x,y)|  \lesssim \frac{|x|^2}{|y|^{n+2}}.$$

Moreover, for $|x|\leq \ve/4$ and $|y|\leq \ve/4$, it is easy to check that $E_i(x,y)$, $i=1,2$, vanishes.
Then, integrating the estimates for $E_i(x,y)$ with respect to $\mu$ and $y$, one
gets $|E(x)|\lesssim |x|^2/\ve^2.$

Assume now that $\supp(\vphi)\subset [1/4,4]$ and take $v\in \R^{n+1}$. To estimate $|E(x)\cdot v|$ 
we may assume that $y\in A(0,\,\ve/4,\,3\ve)$ because otherwise $E_i(x,y)=0$, for $i=1,2$. We get
$$|E_1(x,y)\cdot v| \lesssim \frac{|x|}{\ve^{n+3}}\bigl(\ve\,|x\cdot v| + |x||y\cdot v|\bigr).$$
Concerning $E_2(x,y)$, we have
$$|E_2(x,y)\cdot v| \lesssim \frac1{\ve^{n+5}}\,|g''(\xi_{x,y}| (|x|^2-2x\cdot y)^2 |(x-y)\cdot v|
\lesssim \frac{|x|^2}{\ve^{n+3}}\,\bigl( |x\cdot v| + |y\cdot v|\bigr).$$
Integrating with the preceding inequalities with respect to $\mu$ and $y\in A(0,\,\ve/4,\,3\ve)$, we obtain~\rf{eqev}.
\end{proof}

We will also need the following result. See \cite[Lemma 5.8]{DS1} for the (easy) proof.

\begin{lemma} \label{lemli}
Given $Q\in\DD$, there are $n+1$ points $x_0,\ldots,x_n$ in $Q$ such that $\dist(x_j,L_{j-1})\geq C_7^{-1}\ell(Q)$, where $L_k$ denotes
the $k$-plane passing through $x_0,\ldots,x_k$, and where $C_7$ depends only on $n$ and $C_0$.
\end{lemma}

\begin{lemma}\label{lemcas1}
Let $\vphi:[0,+\infty)\to[0,+\infty)$ be a $\CC^2$ function with $\supp(\vphi)\subset [1/4,\,+\infty)$ and 
$\supp(\vphi')\subset [1/4,\,4]$. Suppose also that $\vphi$ is non decreasing and that
$\chi_{[1/3,3]}\leq C_8\vphi'$.
Let $Q\in\DD$ and $x_0,\ldots,x_n\in Q$ be like in Lemma \ref{lemli}.
Denote $r=\diam(Q)$ and let $\ve> 4 r$.
Suppose that $A(x_0,\ve/\sqrt2,\sqrt2\ve)\cap \supp(\mu)\neq\varnothing$.
Then any point $x_{n+1}\in 3Q$ satisfies
$$\dist(x_{n+1},L_0) \lesssim \ve \sum_{j=1}^{n+1} |R_{\vphi,\ve}\mu(x_j) - R_{\vphi,\ve}\mu(x_0)|+ 
\frac{r^2}{\ve},$$
where $L_0$ is the $n$-plane passing through $x_0,\ldots,x_n$.
\end{lemma}

\begin{proof}
Without loss of generality we assume that $x_0=0$. We denote  by $z$ the orthogonal projection of $x_{n+1}$ onto $L_0$.
Then by Lemma \ref{lemtaylor} we have
\begin{equation}\label{eqlf1}
|T(x_j)|\lesssim |R_{\vphi,\ve}\mu(x_j)- R_{\vphi,\ve}\mu(x_0)| +  \frac{r^2}{\ve^2}.
\end{equation}
for $j=1,\ldots,n+1$.
Let $e_1,\ldots,e_n$ be an orthonormal basis of $L_0$, and set $e_{n+1} =
(x_{n+1}-z)/|x_{n+1}-z|$ (we suppose that $x_{n+1}\not\in L_0$), so that $e_{n+1}$ is 
a unitary vector orthogonal to $L_0$.
Since the points $x_j$, $j=1,\ldots,n$ are linearly independent with ``good constants'' we get 
$$|T(e_i)|\lesssim \frac1r \sum_{j= 1}^n|T(x_j)|
\lesssim \frac1r \sum_{j= 1}^n |R_{\vphi,\ve}\mu(x_j)- R_{\vphi,\ve}\mu(x_0)| + \frac{r}{\ve^2}
$$
for $i= 1,\ldots,n$. Also, since $z\in L_0$ and $|z|\lesssim r$, we have
$|T(z)|\lesssim \sum_{j= 1}^n|T(x_j)|,$
and so by~\rf{eqlf1} (with $j=n+1$),
\begin{equation}\label{eqvv1}
|T(e_{n+1})| = \frac1{\dist(x_{n+1},L_0)}\,|T(z-x_{n+1})| \lesssim \frac1{\dist(x_{n+1},L_0)}\Bigl(
 \sum_{j= 1}^{n+1} |R_{\vphi,\ve}\mu(x_j)- R_{\vphi,\ve}\mu(x_0)| + \frac{r^2}{\ve^2}\Bigr).
\end{equation}
Therefore,
\begin{equation}\label{eqt5}
\Bigl|\sum_{j=1}^{n+1}T(e_j)\cdot e_j\Bigr| \lesssim \frac1{\dist(x_{n+1},L_0)}\,\Bigl(
 \sum_{j= 1}^{n+1} |R_{\vphi,\ve}\mu(x_j)- R_{\vphi,\ve}\mu(x_0)|  + \frac{r^2}{\ve^2}\Bigr).
\end{equation}

On the other hand, from the definition of $T$ in \rf{eqte2} if we denote $y_{(i)} = y\cdot e_i$, we get
\begin{align}\label{eqt55}
\sum_{j=1}^{n+1}T(e_j)\cdot e_j & =
\int\frac1{|y|^{n+1}}\,\biggl[
\vphi\biggl(\frac{|y|^2}{\ve^2}\biggr) \frac{(n+1)\sum_{i> n+1}  y_{(i)}^2}{|y|^2} + 
\vphi'\biggl(\frac{|y|^2}{\ve^2}\biggr) \frac{2\sum_{i=1}^{n+1}  y_{(i)}^2}{\ve^2}
\biggr] d\mu(y)\\
& \gtrsim \frac1{\ve^{n+3}} \,\inf_{t\in [2/5,5/2]}\bigl(\varphi(t),\varphi'(t)\bigr) 
\int_{A(0,(2/5)^{1/2}\ve,(5/2)^{1/2}\ve)
}\sum_{i=1}^{d}y_{(i)}^2\,d\mu(y)  \nonumber\\
& \gtrsim \frac1{\ve^{n+3}}  \int_{A(0,(2/5)^{1/2}\ve,(5/2)^{1/2}\ve)}|y|^2\,d\mu(y) \gtrsim \frac1\ve, \nonumber
\end{align}
since $A(0,\ve/2^{1/2},2^{1/2}\ve)\cap\supp(\mu)\neq\varnothing$.
The lemma follows from \rf{eqt5} and \rf{eqt55}.
\end{proof}

\begin{remark}
For $m\geq p$, set $R_{m,p}\mu = \sum_{k=p}^m R_k\mu$ and $R_{(m)}\mu = \sum_{j\geq m}R_j \mu$.
Take $\ve=2^{-m}$ and $\ve_2 = 2^{-p}>\ve$, and
$\vphi(t) = 1 - \psi(t)$, where $\psi$ is the function introduced in Definition \ref{defpsifi}. 
Under the assumptions and notation of Lemma \ref{lemcas1},
$$\dist(x_{n+1},L_0) \lesssim \ve \sum_{j=1}^{n+1} \biggl(\sum_{k= p}^m |R_k\mu(x_j) - R_k\mu(x_0)|+ 
|R_{(p)}\mu(x_j) - R_{(p)}\mu(x_0)|\biggr) +
\frac{r^2}{\ve}.$$
It is easy to check that
$$|R_{(p)}\mu(x_j) - R_{(p)}\mu(x_0)|\lesssim \frac{r}{\ve_2}.$$
As a consequence, we get
\begin{equation}\label{eqrem}
\dist(x_{n+1},L_0) \lesssim \ve \sum_{j=1}^{n+1} \sum_{k= p}^m|R_k\mu(x_j) - R_k\mu(x_0)|
+ \frac{r^2}{\ve} + \frac{r\ve}{\ve_2}.
\end{equation}
\end{remark}

\begin{lemma}\label{lemcas2}
Let $\vphi:[0,+\infty)\to[0,+\infty)$ be a $\CC^2$ function supported in $[1/4,\,4]$ such that
$\chi_{[1/3,3]}\leq C_8\vphi$.
Let $Q\in\DD$ and $x_0,\ldots,x_n\in Q$ be like in Lemma \ref{lemli}.
Let $r=\diam(Q)$ and $\ve$ such that $2^{m-1}r<\ve\leq 2^m r$, with $m>4$ big enough.
Suppose that $A(x_0,\ve/\sqrt{2},\sqrt{2}\ve)\cap \supp(\mu)\neq\varnothing$, and also that $\dist(x_i,L_Q)\leq C_9\beta_2(Q)$ for $i=0,\ldots,n$, where $L_Q$ is the $n$-plane that minimizes $\beta_2(Q)$.
There exists some constant $\delta_0$ (depending only on $n,d,C_0,C_8,C_9$ and 
$\|\vphi^{(k)}\|_\infty$, $k=0,1,2$) such that if
$$\sum_{k=0}^m \beta_2(B(x_0,2^kr))\leq \delta_0,$$
then any point $x_{n+1}\in 3Q$ satisfies
$$\dist(x_{n+1},L_0) \lesssim \ve \sum_{j=1}^{n+1} |R_{\vphi,\ve}\mu(x_j) - R_{\vphi,\ve}\mu(x_0)|+ 
\frac{r^2}{\ve^{n+2}} \int_{B(x_0,2\ve)}\dist(y,L_0)d\mu(y),$$
where $L_0$ is the $n$-plane passing through $x_0,\ldots,x_n$.
\end{lemma}

\begin{proof}
The proof is similar in part to the one of Lemma \ref{lemcas1}. Like in Lemma \ref{lemcas1}
we assume that $x_0=0$ and we denote  by $z$ the orthogonal projection of $x_{n+1}$ onto $L_0$.
We denote $v= (x_{n+1}-z)/|x_{n+1}-z|$ (we suppose that $x_{n+1}\not\in L_0$), and we set $T^v(x) =T(x)\cdot v$,
$E^v(x) =E(x)\cdot v$, where $T(x)$ and $E(x)$ are defined in Lemma \ref{lemtaylor}. This lemma tells us that
\begin{equation}\label{eqlf2}
|T^v(x_j)|\lesssim |R_{\vphi,\ve}\mu(x_j)- R_{\vphi,\ve}\mu(x_0)| +  |E^v(x_j)|.
\end{equation}
for $j=1,\ldots,n+1$.
Arguing like in \rf{eqvv1} we deduce
\begin{equation}\label{eqtvv}
|T^v(v)| = \frac1{\dist(x_{n+1},L_0)}|T^v(z-x_{n+1})| \lesssim \frac1{\dist(x_{n+1},L_0)}
 \sum_{j= 1}^{n+1} \bigl(|R_{\vphi,\ve}\mu(x_j)- R_{\vphi,\ve}\mu(x_0)| + 
|E^v(x_j)|\bigr).
\end{equation}
Moreover, by \rf{eqev} we have
\begin{equation}\label{eqevv}
\sum_{j= 1}^{n+1} |E^v(x_j)| \lesssim \frac{r\dist(x_{n+1},L_0)}{\ve^2} + \frac{r^2}{\ve^{n+3}} \int_{B(0,2\ve)}\dist(y,L_0)\,d\mu(y),
\end{equation}
since $|y\cdot v|\leq\dist(y,L_0)$.

Now we need to estimate $T^v(v)$ from below. By the definition of $T$ in \rf{eqte2} we have
\begin{align*}
T^v(v) & = \int\frac1{|y|^{n+1}}\,\biggl[
\vphi\biggl(\frac{|y|^2}{\ve^2}\biggr) \biggl(1-\frac{(n+1)(y\cdot v)^2}{|y|^2}\biggr) + 
\vphi'\biggl(\frac{|y|^2}{\ve^2}\biggr) \frac{2(y\cdot v)^2}{\ve^2}
\biggr] d\mu(y)\\
 & \geq \int
\vphi\biggl(\frac{|y|^2}{\ve^2}\biggr) \frac1{|y|^{n+1}}d\mu(y) -(n+1) \int \vphi\biggl(\frac{|y|^2}{\ve^2}\biggr)\frac{\dist(y,L_0)^2}{|y|^{n+3}}
 d\mu(y)\\
&\quad -2\int \vphi'\biggl(\frac{|y|^2}{\ve^2}\biggr) \frac{\dist(y,L_0)^2}{\ve^2|y|^{n+1}} d\mu(y).
\end{align*}
Recall that $A(0,\ve/\sqrt{2},\sqrt{2}\ve)\cap \supp(\mu)\neq\varnothing$, and so
$$\int
\vphi\biggl(\frac{|y|^2}{\ve^2}\biggr) \frac1{|y|^{n+1}}d\mu(y) \gtrsim \frac1\ve.$$
Then we infer that
$$T^v(v) \geq \frac{C_{10}}{\ve} - C_{11}\int_{B(0,2\ve)}\frac{\dist(y,L_0)^2}{\ve^{n+3}}\,d\mu(y).$$
For $y\in B(0,2\ve)$, we have
$$\dist(y,L_0) \lesssim \dist(y,L_{B(0,2\ve)}) + \dist_H(L_0,L_{B(0,2\ve)}) \lesssim
\dist(y,L_{B(0,2\ve)}) + \ve\sum_{k=1}^m \beta_2(B(0,2^kr)),$$
where $L_{B(0,2\ve)}$ stands for the $n$-plane which minimizes $\beta_2(B(0,2\ve))$.
Therefore, we get
$$\int_{B(0,2\ve)}\frac{\dist(y,L_0)^2}{\ve^{n+3}}\,d\mu(y) \lesssim \frac1\ve\Bigl(\sum_{k=1}^m \beta_2(B(0,2^kr))\Bigr)^2 \lesssim \frac{\delta_0^2}\ve,
$$
and then,
$$T^v(v) \geq \frac{C_{10}}{\ve} - C_{12}\frac{\delta_0^2}\ve.$$
As a consequence, if $\delta_0$ is small enough,
$T^v(v)\gtrsim 1/\ve$. From this estimate, \rf{eqtvv}, and \rf{eqevv} we deduce that
$$\dist(x_{n+1},L_0) \lesssim \ve \sum_{j=1}^{n+1} |R_{\vphi,\ve}\mu(x_j) - R_{\vphi,\ve}\mu(x_0)|+ 
\frac{r^2}{\ve^{n+2}} \int_{B(x_0,2\ve)}\dist(y,L_0)d\mu(y)+ \frac{r\dist(x_{n+1},L_0)}{\ve}.$$
If $\ve/r$ is big enough, the lemma follows.
\end{proof}


\subsection{Proof of Theorem \ref{teoqo}}

The second statement of the theorem is a direct consequence of \rf{eqteoqo}. So we only have to deal with
\rf{eqteoqo}. To prove it, first we will estimate $\beta_2(P)$ for any $P\subset Q$, $P\in\DD_k$, with $P$ small enough.
To this end, take points $y_0,\ldots,y_n$ in $P$ as in Lemma \ref{lemli} and set $\ve= 2^{-m}$, with $\ve\gg\diam(P)$
to be fixed below.

Let $r=\diam(P)$. It is easy to check that there exists some constant $0<C_{13}<1$ small enough such that
any collection of points $x_0,\ldots,x_n$ with  $x_j\in B(y_j,C_{13}r)$, $j=0,\ldots,n$, also satisfies the conditions
of Lemma \ref{lemli} (maybe with some constant somewhat bigger that $C_7$). Consider the sets
$$G_j = \bigl\{x\in B(y_j,C_{13}r):\,\dist(x,L_P)\leq C_{14}\beta_2(P) \ell(P)\bigr\}\qquad j=0,\ldots,n,$$
where $L_P$ is the $n$-plane that minimizes $\beta_2(P)$.
By Chebyshev, if $C_{14}$ is chosen big enough, $\mu(G_j)\geq \mu(B(y_j,C_{13}r))/2$ for all $j$.

We distinguish two cases:

\vvv
\noi {\bf 1) } Suppose that $\sum_{m+1\leq i\leq k} \beta_2(B(z,2^{-i})) \leq \delta_0$
for any $z\in G_0$.

Take $x_0,\ldots,x_n$ so that $x_j\in G_j$ for each $j$. 
Notice that if $2^{m_0}>C_0^2$, then $A(x_0,2^{-i},2^{-i+1})\cap \supp(\mu)\neq\varnothing$ for some $i$ with
$m\leq i\leq m+m_0$ (here we need to assume that $\ve\lesssim \diam(\supp(\mu))$, which is true if $P$ is small enough).
Then, by Lemma \ref{lemcas2}, any point $x_{n+1}\in 3P$ satisfies
\begin{equation}\label{eqk12}
\dist(x_{n+1},L_0) \lesssim \ve \sum_{i=m}^{m+m_0} \sum_{j=1}^{n+1} |R_i\mu(x_j) - R_i\mu(x_0)|+ 
\frac{r^2}{\ve^{n+2}} \int_{B(x_0,2^{m_0+2}\ve)}\dist(y,L_0)d\mu(y),
\end{equation}
where $L_0$ is the $n$-plane passing through $x_0,\ldots,x_n$, and the constant in $\lesssim$ may depend on $m_0$.

Let $\wh{P}$ be the smallest ancestor of $P$ such that $3\wh{P}$ contains $B(x_0,2^{m_0+4}\ve)$ for all the points 
$x_0\in G_0$. Clearly, $\ell(\wh{P})\approx 2^{-m}=\ve$. 
Because $x_j \in G_j$ for $0\leq j \leq n$, it is easy to check that for all $y\in \wh{P}$,
\begin{align*}
\dist(y,L_0) &
\lesssim \dist(y,L_{\wh{P}}) +
d_H\bigl(L_{\wh{P}} \cap B(x_0,\diam(\wh{P})), L_0 \cap B(x_0,\diam(\wh{P}))\bigr) \\
& \lesssim \dist(y,L_{\wh{P}}) + \ve\sum_{M\in \DD: P\subset M\subset \wh{P}}  \beta_2(M).
\end{align*}
Therefore,
$$\int_{B(x_0,2^{m_0+2}\ve)}\dist(y,L_0)d\mu(y) \lesssim \ve^{n+1} \sum_{M\in \DD: P\subset M\subset \wh{P}}  
\beta_2(M).$$
By \rf{eqk12} and the preceding estimate, we get
\begin{equation}\label{eqk22}
\frac{\dist(x_{n+1},L_0)^2}{r^2} \lesssim \frac{\ve^2}{r^2} \sum_{i=m}^{m+m_0} \sum_{j=1}^{n+1} |R_i\mu(x_j) - R_i\mu(x_0)|^2+ 
\frac{r^2}{\ve^{2}} \Bigl(\sum_{M\in \DD: P\subset M\subset \wh{P}}  \beta_2(M)\Bigr)^2.
\end{equation}
Since 
$$\Bigl(\sum_{M\in \DD: P\subset M\subset \wh{P}}  \beta_2(M)\Bigr)^2 \lesssim \log(\ve/r) \sum_{M\in \DD: P\subset M\subset \wh{P}}  \beta_2(M)^2,$$
integrating \rf{eqk22} on $x_j\in G_j$ for $0\leq j \leq n$ and on $x_{n+1}\in 3P$, we obtain
\begin{equation}\label{eqk32}
\beta_2(P)^2\mu(P) \lesssim \frac{\ve^2}{r^2} \sum_{i=m}^{m+m_0} \|R_i \mu\|_{L_2(\mu|3P)}^2 + 
\frac{r^2}{\ve^{2}}\,\log\Bigl(\frac\ve r\Bigr) \sum_{M\in \DD: P\subset M\subset \wh{P}}  \beta_2(M)^2 \mu(P).
\end{equation}

\vvv
\noi {\bf 2) } Suppose now that $\sum_{m+1\leq i\leq k} \beta_2(B(z,2^{-i})) > \delta_0$
for some $z\in G_0$. This implies that
$$\sum_{M\in \DD: P\subset M\subset \wh{P}}  \beta_2(M) \gtrsim \delta_0,$$
where $\wh{P}$ is defined as in the first case.

Also as in case 1), we take $x_0,\ldots,x_n$ with $x_j\in G_j$ for all $j$. 
Recall that $A(x_0,2^{-i},2^{-i+1})\cap \supp(\mu)\neq\varnothing$ for some $i$ with
$m\leq i\leq m+m_0$, assuming that $2^{m_0}>C_0^2$ and that $P$ is small enough.
Then, by Lemma \ref{lemcas1} and the subsequent remark, if set $\ve_2=2^{p_0}\ve$ with $p_0\geq1$, any point $x_{n+1}\in 3P$ satisfies
$$
\dist(x_{n+1},L_0) \lesssim \ve \sum_{j=1}^{n+1} \sum_{k= m-p_0}^{m+m_0}|R_k\mu(x_j) - R_k\mu(x_0)|
+ \frac{r^2}{\ve} + \frac{r\ve}{\ve_2}.
$$
If we take $\ve_2$ such that $r\ve/\ve_2\approx r^2/\ve$, we get
\begin{align*}
\dist(x_{n+1},L_0) & \lesssim \ve \sum_{j=1}^{n+1} \sum_{k= m-p_0}^{m+m_0}|R_k\mu(x_j) - R_k\mu(x_0)|
+ \frac{r^2}{\ve} \\
& \lesssim \ve \sum_{j=1}^{n+1} \sum_{k= m-p_0}^{m+m_0}|R_k\mu(x_j) - R_k\mu(x_0)|
+ \frac{r^2}{\ve\delta_0} \sum_{M\in \DD: P\subset M\subset \wh{P}}  \beta_2(M).
\end{align*}
Operating as in \rf{eqk22} and \rf{eqk32}, we obtain
\begin{equation}\label{eqk56}
\beta_2(P)^2\mu(P) \lesssim \frac{\ve^2}{r^2} \sum_{i=m-p_0}^{m+m_0} \|R_i \mu\|_{L_2(\mu|3P)}^2 + 
\frac{r^2}{\ve^{2}\delta_0^2}\,\log\Bigl(\frac\ve r\Bigr) \sum_{M\in \DD: P\subset M\subset \wh{P}}  \beta_2(M)^2 \mu(P).
\end{equation}

\vvv
Notice that in both cases 1) and 2) the estimate \rf{eqk56} holds for any cube $P$ with $\ell(P)\leq C_{15}\ell(Q)$, where
$C_{15}$ is small enough and depends on $\ve,r$, etc., since $p_0<m$ and $\delta_0<1$. 
Recall that in \rf{eqk56} we have $r\approx\ell(P)$ and $\ve= 2^m\approx\ell(\wh{P})$. Given some constant $0<\tau\ll1$ to be fixed below
and any dyadic cube $P\subset Q$,
we take $\ve$ such that $r\approx\tau\ve$. 
The sum of \rf{eqk56} over $P\subset Q$ such that $\ell(P)\leq C_{15}\ell(Q)$ gives
$$\sum_{\substack{P\subset Q\\ \ell(P)\leq C_{15}\ell(Q)}}\beta_2(P)^2\mu(P) \leq
C(\tau) \sum_{i\in\Z} \|R_i \mu\|_{L_2(\mu|3Q)}^2 + 
\frac{C \tau^2|\log\tau|}{\delta_0^2}\sum_{P\subset Q}\sum_{M: P\subset M\subset \wh{P}}  \beta_2(M)^2 \mu(P).$$
Now we use $|J(P)-J(\wh P)| \approx |\log\tau|,$  as well as the trivial estimate $\beta_2(P)\lesssim 1$ for $\ell(P)> C_{15}\ell(Q)$, and then
we obtain
$$\sum_{\substack{P\subset Q}}\beta_2(P)^2\mu(P) \leq
C(\tau) \sum_{i\in\Z} \|R_i \mu\|_{L_2(\mu|3Q)}^2 + 
\frac{C_{16} \tau^2|\log\tau|^2}{\delta_0^2}\sum_{P\subset Q}  \beta_2(P)^2 \mu(P) + \mu(Q).$$
If we choose $\tau$ such that $C_{16} \tau^2|\log\tau|^2/\delta_0^2\leq 1/2$, the theorem follows.
\fiproof



\begin{thebibliography}{AHMTT}



\bibitem[Ch]{Christ} M. Christ, {\em A $T(b)$ theorem with remarks on analytic capacity and the Cauchy integral,} Colloq. Math. 60/61 (1990), no. 2,
601--628. 



\bibitem[Da1]{David-surfaces} G. David, {\em Op\'{e}rateurs d'int\'{e}grale singuli\`ere sur les surfaces
r\'eguli\`eres,} Ann. scient. \'{E}c. Norm. Sup. (4) 21 (1988), 225--258.

\bibitem[Da2]{David-LNM} G. David, {\em Wavelets and singular integrals on curves
and surfaces,} Lecture Notes in Math. {1465}, Springer-Verlag,
Berlin, 1991.

\bibitem[Da2]{David-revista} G. David, {\em Unrectifiable $1$-sets have vanishing
analytic capacity,} Revista Mat. Iberoamericana 14(2) (1998),
369--479.


\bibitem[DS1]{DS1} G. David and S. Semmes, {\em Singular integrals and
rectifiable sets in $R_n$: Beyond Lipschitz graphs,} Ast\'{e}risque
No. 193 (1991).

\bibitem[DS2]{DS2} G. David and S. Semmes, {\em Analysis of and on uniformly
rectifiable sets,} Mathematical Surveys and Monographs, 38,
American Mathematical Society, Providence, RI, 1993.


\bibitem[Do]{Dorronsoro} J.R. Dorronsoro, {\em A characterization of potential spaces,} Proc. Amer. Math. Soc.
95 (1985), 21--31.

\bibitem[Fa]{Farag} H.\ M.\ Farag, {\em The Riesz kernels do not give rise
to higher-dimensional analogues of the Menger-Melnikov curvature},
Publ. Mat. 43 (1999), no. 1, 251--260.

\bibitem[Jo1]{Jones-Escorial} P.W. Jones, {\em Square functions, Cauchy integrals, analytic capacity, and harmonic
 measure. } Harmonic analysis and partial differential equations (El Escorial, 1987),  24--68, Lecture Notes in Math., 
1384, Springer, Berlin, 1989.

\bibitem[Jo2]{Jones} P.W. Jones, {\em Rectifiable sets and the travelling
salesman problem,} Invent. Math. 102 (1990), 1--15.

\bibitem[L\'e]{Leger} J.C. L\'eger, {\em Menger curvature and
rectifiability,} Ann. of Math. 149 (1999), 831--869.

\bibitem[MT]{MT} J. Mateu and X. Tolsa, {\em Riesz transforms and harmonic
Lip$_1$-capacity in Cantor sets}, Proc. London Math. Soc. 89(3) (2004), 676--696.



\bibitem[Ma]{Mattila-llibre} P. Mattila. {\em Geometry of sets and measures in
Euclidean spaces,} Cambridge Stud. Adv. Math. 44, Cambridge Univ.
Press, Cambridge, 1995.



\bibitem[MMV]{MMV} P. Mattila, M.S. Melnikov and J. Verdera, {\em The Cauchy
integral, analytic capacity, and uniform rectifiability,} Ann. of
Math. (2) 144 (1996), 127--136.



\bibitem[MPr]{MPr} P.\ Mattila and D.\ Preiss, {\em
Rectifiable measures in $\R^n$ and existence of principal values for singular integrals},
J. London Math. Soc. (2) { 52} (1995), no. 3, 482--496.






\bibitem[Pa]{Pajot} H. Pajot, {\em Analytic capacity, rectifiability, Menger curvature
and the Cauchy integral,} Lecture Notes in Math. 1799, 
Springer, 2002.


\bibitem[Se]{Semmes} S. Semmes, {\em Analysis vs. geometry on a class of
rectifiable hypersurfaces in $\R^n$,} Indiana Univ. Math. J. 39
(1990), 1005--1035.

\bibitem[To1]{Tolsa-pubmat2} X. Tolsa, {\em $L^2$ boundedness of the Cauchy transform implies $L^2$ boundedness 
of all Calder\'on-Zygmund operators associated to odd kernels,} Publ. Mat. 48 (2004), no.~2, 445--479.






\bibitem[To2]{Tolsa-sem} X. Tolsa, {\em Painlev\'{e}'s problem and the
semiadditivity of analytic capacity}, Acta Math. 190:1 (2003),
105--149.

\bibitem[To3]{Tolsa-bilip} X. Tolsa, {\em Bilipschitz maps, analytic
capacity, and the Cauchy integral}, Ann. of Math. 162:3 (2005), 1241--1302.

\bibitem[To4]{Tolsa-jfa} X. Tolsa, {\em
Principal values for Riesz transforms and rectifiability}, J. Funct. Anal. 254(7) (2008), 1811--1863.

\bibitem[Vo]{Volberg} A. Volberg, {\em Calder\'on-Zygmund capacities and operators on nonhomogeneous spaces.}
CBMS Regional Conf. Ser. in Math. 100, Amer. Math. Soc., Providence, 2003.


\end{thebibliography}
\end{document}